\documentclass[12pt,A4paper]{article}

\usepackage[left=30mm,right=30mm,top=25mm,bottom=30mm]{geometry}
\usepackage{epsfig}
\usepackage{epstopdf}
\usepackage{soul}
\usepackage{xfrac}
\usepackage{amsmath}
\usepackage{enumerate}

\usepackage[percent]{overpic}

\usepackage{color}
\usepackage{amsthm,amsmath,amssymb}
\usepackage{booktabs}
\usepackage{mathpazo}
\usepackage{microtype}
\usepackage{overpic}
\usepackage{bm}
\usepackage{sectsty}
\usepackage[
	pdftitle={PDFTitle},
	pdfauthor={Hugo Parlier},
	ocgcolorlinks,
	linkcolor=linkred,
	citecolor=linkred,
	urlcolor=linkblue]
{hyperref}

\usepackage{tikz}
\usetikzlibrary{calc,decorations.pathreplacing}

\definecolor{linkred}{RGB}{157,91,246} %DeepSkyBlue
\definecolor{linkblue}{RGB}{16, 78, 139}

\usepackage[hang,flushmargin]{footmisc}
\usepackage{titlesec}
	\titlespacing{\section}{0pt}{12pt}{0pt}
	\titlespacing{\subsection}{0pt}{6pt}{0pt}
	
\titlelabel{\thetitle.\quad}

\theoremstyle{plain}
\newtheorem{theorem}{Theorem}[section]
\newtheorem{proposition}[theorem]{Proposition}
\newtheorem{lemma}[theorem]{Lemma}
\newtheorem{corollary}[theorem]{Corollary}

\theoremstyle{definition}

\newtheorem{remark}[theorem]{Remark}

\newcommand{\sys}{{\rm sys}}

\newcommand{\R}{{\mathbb R}}

\newcommand{\Hyp}{{\mathbb H}}

\newcommand{\C}{{\mathcal C}}

\newcommand{\G}{{\mathcal G}}
\newcommand{\T}{{\mathcal T}}
\newcommand{\M}{{\mathcal M}}
\newcommand{\nf}{{\rm nf}}

\newcommand{\area}{{\rm area}}
\newcommand{\arcsinh}{{\,\rm arcsinh}}
\newcommand{\arccosh}{{\,\rm arccosh}}
\newcommand{\arctanh}{{\,\rm arctanh}}
\newcommand{\diam}{{\rm diam}}

\newcommand{\length}{\ell}

\newcommand{\E}{\mathcal E}

\newcommand{\NN}{\mathcal N}

\sectionfont{\large \bfseries}
\subsectionfont{\normalsize}

\long\def\symbolfootnote[#1]#2{\begingroup%
\def\thefootnote{\fnsymbol{footnote}}\footnote[#1]{#2}\endgroup}

\def\blfootnote{\xdef\@thefnmark{}\@footnotetext}

\usepackage{mathtools}

\usepackage[sort,nocompress]{cite}

\begin{document}

{\Large \bfseries The entropy spectrum of hyperbolic surfaces}

{\large Ara Basmajian\symbolfootnote[1]{\small 
Supported by a Dolciani faculty research grant, a PSC-CUNY research grant, and a grant from the Simons Foundation (TSM 00013865 A.B.).},
Hugo Parlier\symbolfootnote[7]{\small Supported by ANR-SNF Grant number 200021E\_238147
(SUGAR).
}
}

\vspace{0.5cm}
{\bf Abstract.}
This article introduces and studies the entropy spectrum of a hyperbolic surface, that is the set of entropies of its subsurfaces. The main results are that the entropy spectrum is a reverse well-ordered multiset, with finite multiplicities, and that there is a quantifiable gap around the value $1$. This gap comes from a counting result on the number of non-filling geodesics which in turn comes from explicit estimates on the number of curves on surfaces with boundary in terms of geometric data. The geometric data includes lengths of boundary geodesics and so-called boundary width, which measures maximal distance to the boundary but can also be interpreted in terms of the topology of the surface and the systole. 

\vspace{0.5cm}

\section{Introduction} \label{sec:intro}

Growth questions about the number of closed geodesics on hyperbolic surfaces up to a given length have a rich and illustrious history. These matters are intimately related to the understanding of dynamical properties of surfaces and to relationships with spectral geometry. Whereas growth in the length spectrum is more about the dynamics of the surface, growth in the simple length spectrum is more about dynamical properties of the underlying moduli space. 

For orientable and closed surfaces, Huber famously found an asymptotic formula for the growth rate of closed geodesics. More generally, for any finite type hyperbolic surface $X$, the growth rate of the number of (oriented, primitive) closed geodesics of length up to $L$ is asymptotic to
\[
\frac{e^{hL}}{hL}
\]
where $h=h_X$ is the topological or volume entropy of $X$ (and these are equal). The volume entropy quantifies the area growth of balls in the universal cover of the surface. For instance, if $X$ is closed or of finite area, $h=1$, but if $X$ has at least one boundary geodesic of positive length, then $0<h<1$. This beautiful relationship between the growth of curves and entropy is further enhanced by a direct relationship to the Hausdorff dimension of the limit set. 

In particular, this means that for a closed surface $X$, each of its proper subsurfaces has an entropy which lies in the interval $]0,1[$. This collection of entropies gives rise to a multiset $\E(X)$ which we call the {\it entropy spectrum} and which is the main object of focus of this paper. 

To study the entropy spectrum requires studying subsurfaces which necessarily have bounday geodesics. Hence, we begin by studying orientable compact hyperbolic surfaces with non-empty boundary, so finite-type surfaces with at least one simple closed geodesic as boundary. Our first main result is an explicit bound on the number of its closed geodesics in terms of two pieces of geometric data: its total boundary length and its boundary width, i.e., the maximal distance between a point on the surface and the boundary.

\begin{theorem}\label{thm:mainestimate}
Let $X$ be a hyperbolic surface with boundary length $B>0$ and boundary width $W$. Then 
\[
N_X(L) \leq C \,e^{\,\lambda L}
\]
where $\lambda$ and $C$ are explicit functions depending on $B$ and $W$ and $\lambda(B,W) < 1$. 
\end{theorem}

In particular, the topological entropy of $X$ satisfies $h_X \leq \lambda(B,W)$. Explicit versions of the constants are given in Remark \ref{rem:explicit}. 

The tools we use to show this are inspired by a body of work initiated by Thurston in Teichm\"uller theory \cite{ThurstonStretch}. Thurston was interested in how lengths of curves could be used to quantify proximity between surfaces, showing that ratios of simple curve lengths determines the optimal Lipschitz constant taken among homeomorphisms between the two surfaces. In this paper, so-called strip map deformations were introduced, and have since been used in a variety of settings, namely in the work of Danciger, Gu\'eritaud and Kassel \cite{DGK}. 

We use and quantify strip maps to find deformations which affect the lengths of all closed curves. We show that, given a hyperbolic surface $X$, we can find a comparison surface $X'$ with cusps, where the lengths are all reduced by a multiplicative factor. We then adapt a result of Buser to obtain the result.

Our result can also be reinterpreted in terms of other types of geometric data, such as diameter:

\begin{corollary}\label{cor:diam}

Let $X$ be a hyperbolic surface with boundary geodesics and of diameter $\diam(X)$. Then
\[
N_X(L) \leq C\, e^{ \,\eta L}
\]
where $C$ and $\eta$ are explicit functions which only depend on $\diam(X)$ and such that $\eta(\diam(X))< 1$. 
\end{corollary}

While the above results all apply to surfaces with non-empty geodesic boundary, in what follows, we focus on closed surfaces. Our first main application is a result about non-filling geodesics. It follows from Lalley's equidistribution theorem for periodic orbits \cite{LalleyPeriodic}, interpreted in terms of geodesic currents \cite{BonahonCurrents} as in \cite[\S2]{AougabSouto}, that asymptotically almost every closed geodesic is filling.

In particular this means that the growth of the number of filling closed geodesics, on a finite-type complete hyperbolic surface, is asymptotic to $e^L/{2L}$ as $L$ grows. Our next results shows that there is an explicit bound, which depends on the geometry of the surface, on the number of its non-filling geodesics. 

\begin{theorem}\label{cor:nf}
There exist constants $a=a(g,\varepsilon)<1$ and $A=A(g,\varepsilon)$ such that
\[
N^{\nf}_X(L) \leq A \, e^{ \,a L}
\]
for any closed hyperbolic surface $X$ of genus $g$ and systole at least $\varepsilon$. 
\end{theorem}

The corresponding statement for finite-area surfaces can also be obtained by treating the cusps separately, but for simplicity, we restrict to closed surfaces here. This result uses the following observation: if a curve is non-filling, then there is a simple closed geodesic on its complement. Hence, to bound the number of non-filling geodesics, it suffices to bound the number of simple closed geodesics and, for each of them, use Theorem \ref{thm:mainestimate} to bound the number of curves on their complement. 

In particular, the above corollary immediately proves that there is a first, quantifiable gap in the entropy spectrum:

\begin{corollary}\label{cor:gap}
For any closed hyperbolic surface $X$ of genus $g$, we have 
\[
\sup_{Y \subsetneq X} h(Y) < 1
\]
where $Y$ is taken among all proper subsurfaces, and this supremum is bounded above by a function of the geometry of $X$. In particular, there is a quantifiable gap between $1$ and the other entropies in $\E(X)$.
\end{corollary}
In fact, a more precise version is shown (Corollary \ref{cor:gapexplicit})) where the supremum is shown to be bounded above by a function of the genus and the systole of $X$. Theorem \ref{thm:welldone} below will show that the supremum above is in fact a maximum.

We now move to structural results about the entropy spectrum $\E(X)$ of a closed surface $X$. For the purpose of describing accumulation points, we adopt the convention that the empty subsurface has entropy $0$. Our main structural result is the following. 

\begin{theorem}\label{thm:welldone}
The entropy spectrum is a reverse well-ordered multiset. Its accumulation points correspond exactly to the entropies of connected subsurfaces $Y$ such that $X\setminus Y$ has curve complexity at least $1$. 
\end{theorem}

In particular, the reverse well-ordered result implies that there is a largest value in $\E(X)$ which is not $1$, and which we denote by $h_*(X)$ and call the {\it largest proper entropy}. Our results show that in fact this value is realized as the entropy of a subsurface obtained by taking the complement of a non-separating simple closed geodesic (see Corollary \ref{cor:proper}). The quantity $1-h_*(X)>0$ can be thought of as an entropy gap. We observe that, using results of Lenzhen and Souto \cite{LenzhenSouto}, it is exactly half the Hausdorff codimension of the set of complete non-filling geodesics (see Remark \ref{rem:LS}).

While the results and methods are very different, one motivation for the above theorem lies in the work of Fujiwara and Sela on growth of subgroups of hyperbolic groups and relatives \cite{FujiwaraSela2023,Fujiwara2025}. In their context, by varying the generating set, they show that the orders of growth are well-ordered, which is seemingly the opposite result to ours. While this might seem surprising, this is mainly because the limiting phenomenon is opposite: group growth degenerations approximate the limit from below, whereas in our context, subsurfaces approximate entropies from above. 

\noindent {\bf Organization.}
This paper is organized as follows. In Section \ref{sec:setup}, we setup notation, recall useful results, define the entropy spectrum, and prove marked rigidity of the entropy spectrum (Theorem \ref{thm:marked}). In Section \ref{sec:bound} we prove Theorem \ref{thm:mainestimate} and consequences, including Theorem \ref{cor:nf}. This also requires a technical boundary reduction result (Lemma \ref{lem:boundary-reduction}). In Section \ref{sec:entropy}, we prove the main structural results about the entropy spectrum, namely Theorem \ref{thm:welldone}. 

\noindent {\bf Acknowledgements.}

We are grateful to all the people we discussed this work with and who made numerous suggestions and helpful comments. In particular, thanks to Marie Abadie, Nalini Anantharaman, Viveka Erlandsson, Federica Fanoni, Jonah Gaster, Didac Martinez-Granado and Juan Souto. We also thank the Bernoulli center (EPFL) for their support during the writing of this paper.

\section{Definitions, setup and preliminary results}\label{sec:setup}

In addition to setting notation, in this section we define the main objects of study and state some preliminary results. The only new result is the proof of the marked rigidity of the entropy spectrum (Theorem \ref{thm:marked}).

\subsection{Curves, geodesics and orthogeodesics}\label{ss:curves}
We consider connected, orientable hyperbolic surfaces $X$ of finite type with boundary consisting in simple closed geodesics or cusps. Often the genus of $X$ is denoted $g$ and its number of boundary components by $n$. The Euler characteristic of $X$ is then $\chi(X)=\chi = 2-2g-n$ and, because we suppose $X$ to be hyperbolic, $\chi(X) <0$. 

By curve we mean a closed curve, which we will always suppose to be non-trivial, meaning homotopically non-trivial and non-parallel to a single boundary element, and primitive, meaning not freely homotopic to a $k$-iterate of another closed curve for $k\geq 2$. We consider curves as non-oriented, and we readily use the fact that a curve on a hyperbolic surface is freely homotopic to a unique closed geodesic, which in turn is the unique length minimizer in the free homotopy class. For a curve $\gamma$ on a hyperbolic surface $X$, we denote by $\ell_X(\gamma)$ the length of this unique closed geodesic (or simply by $\ell(\gamma)$ if the surface $X$ is clear from the context). 

The curve complexity of a surface is its maximal number of non-peripheral disjoint simple closed curves. This is the number of interior curves in a pants decomposition. Thus a surface has curve complexity at least $1$ if one of its connected components contains a non-peripheral essential simple closed curve. More generally, it has curve complexity at least $k$ if one of its components is of type $(g,n)$ with $3g-3+n\geq k$.

The systole of $X$ (or interior systole of $X$) is the length $\sys(X)$ of the shortest interior closed geodesic. A curve is simple if it is homotopic to a simple closed curve, and if this is the case, its geodesic representative is also simple. If $X$ is not a pair of pants (meaning if $X$ is not homeomorphic to three holed sphere) then a systolic geodesic (often called systole by abuse of notation) is simple. On a pair of pants it is always a so-called figure-$8$ curve, that is a curve with one self-intersection point. When $X$ is closed, it is of genus $g\geq 2$ and we denote by $\M_g$ the moduli space of all hyperbolic structures up to isometry. The $\varepsilon$-thick part of moduli space, consisting of surfaces in $\M_g$ of systole at least $\varepsilon>0$, is denoted $\M_g^{\varepsilon}$. Similarly, we denote by $\T_g$ the Teichm\"uller space of marked hyperbolic surfaces up to isotopy. 

On a surface with geodesic boundary, an orthogeodesic is a geodesic arc orthogonal to the boundary in both of its endpoints. Like for curves, each free homotopy class of arc with endpoints on the boundary (but not fixed by the homotopy), has a unique orthogeodesic representative which again is the unique length minimizer. As for curves and closed geodesics, we denote by $\ell_X(\mu)$ the length of an orthogeodesic $\mu$. 

The collar lemma states that around a simple closed geodesic $\gamma$, there is an embedded collar of width 
\[
\arcsinh \left( \frac{1}{\sinh\left(\frac{\ell(\gamma)}{2}\right)} \right).
\]
If $\gamma$ is a boundary geodesic, only one side of the collar is embedded. The area of a half-collar around a simple closed geodesic $\gamma$, which can be a boundary curve, of width $r$ is given by the formula 
\[
\ell(\gamma) \sinh r.
\]
In particular, if $X$ has boundary length $\ell(\partial X)$, then if $r$ satisfies 
\[
r\geq \arcsinh\left( \frac{1}{\ell(\partial X)} \area(X) \right)
\]
then the $r$ collar around $\partial X$ is not embedded, otherwise the area of the $r$-neighborhood of the boundary would exceed the area of the surface. This implies that $X$ always has a orthogeodesic of length less than twice this quantity. Said otherwise, and using the terminology of \cite{BasmajianFanoni}, an orthosystole $\mu_0$ (the, or a, shortest orthogeodesic of $X$) satisfies
\[
\ell(\mu_0) < 2 \arcsinh\left( \frac{1}{\ell(\partial X)} \area(X) \right).
\]
A more precise bound can be found in the work of C. Bavard \cite{Bavard} who has tight bounds on the length of the orthosystole among all surfaces with fixed boundary length, generalizing previous work of Schmutz Schaller for surfaces with cusps \cite{Schmutz}. 

There is also a collar lemma for simple orthogeodesics. One way to see this is by doubling the surface along its boundary, a procedure which transforms simple orthogeodesics to simple closed geodesics. There is a embedded strip around an orthogeodesic $\mu$ of collar width 
\[
\arcsinh \left( \frac{1}{\sinh(\ell(\mu))} \right).
\]
Note that, unlike in the corresponding formula for simple closed geodesics, the length is not divided by two (this follows from the doubling argument stated above). Importantly, if two simple orthogeodesics are disjoint, their corresponding strips are disjoint.

The results stated above are standard facts on the hyperbolic geometry of surfaces, and can often be proved using hyperbolic trigonometry. We refer to Buser's book \cite{BuserBook} for specific formulas, proofs and references for the different collar lemmas. When, in the sequel, we say that we are using standard hyperbolic trigonometry, more often than not. we using a formula from \cite[p. 454]{BuserBook}.

\subsection{Surfaces with boundary geodesics and length expansion}

Let $X$ be a compact hyperbolic surface with nonempty geodesic boundary. Its
\emph{width} is
\[
 w(X)=\max_{x\in X} d(x,\partial X).
\]
The width can be estimated using other invariants, such as systole length and area.

Let $\varepsilon$ be the minimum length among all closed geodesics of $X$, including boundary geodesics. In other words, $\varepsilon$ is the min of the (interior) systole length and the length of a smallest boundary component. Now by an area argument, we have:
\begin{equation}\label{eq:systolewidth}
w(X) \leq \frac{1}{2 \sinh \frac{\varepsilon}{4}} \,\area(X)
\end{equation}
This is because the $\frac{\varepsilon}{4}$ strip around a distance path from a point to $\partial X$ is embedded on $X$, and its area, equal to $2 w(X) \sinh \frac{\varepsilon}{4}$ for a point at distance $w(X)$, is less than the total area of the surface. 

The following (well-known) result is useful to keep in mind, namely in terms of the proof strategy, There are multiple sources for it and its uses (see for instance \cite{ThurstonSpine}, \cite{DGK}, \cite{Gueritaud2016}, \cite{ParlierLengths} and \cite{Papadopoulos-Theret}).

\begin{lemma}[Length expansion lemma]\label{lem:lengthexpansion} Let $X$ be a finite-type hyperbolic surface with boundary geodesics of length $(\ell_1,\hdots,\ell_n)$ and let $\varepsilon>0$. Then there exists a hyperbolic surface $X' \cong X$ with boundary geodesics of length $(\ell_1+\varepsilon,\hdots,\ell_n)$ and such that any non-trivial simple closed curve $\gamma\subset \Sigma$ satisfies
\[
\ell_{X'}(\gamma)>\ell_{X}(\gamma).
\]
\end{lemma}
The proof of Theorem \ref{thm:mainestimate} quantifies the ratio by which every curve must decrease. 

\subsection{Counting curves}

In the case of closed surfaces or surfaces with cusps, there is an immediate effective upper bound on the number of primitive closed geodesics of length less than $L$. Denote by $N_X(L)$ the number of (non-oriented) primitive closed geodesics of length less than $L$.

The following lemma is just a slight adaptation of Lemma 6.6.4 from \cite{BuserBook}. The original argument, for closed surfaces, is to use a type of $\varepsilon$-net of the thick part of a surface and to then count the number of disjoint balls in the universal cover. Here we provide a quick argument to show how to adapt it to surfaces with cusps.

\begin{lemma}\label{lem:buser}
Let $X$ be a hyperbolic surface of genus $g$ with $n$ cusps. Then 
\[
N_X(L) \leq (g+ \lfloor n/2 \rfloor) \,e^{L+6}.
\]
\end{lemma}

\begin{proof}
If $n$ is even, by identifying cusps in pairs, one obtains a noded surface of genus $g + n/2$. If $n$ is odd, by adding a once-punctured torus, we can obtain a noded surface of genus $g + \lfloor n/2 \rfloor + 1$. In either case, we can apply Lemma 6.6.4 from \cite{BuserBook} to obtain the result.
\end{proof}

As mentioned in the introduction, Huber \cite{Huber} proved an asymptotic formula for the number of primitive closed geodesics on finite area geodesically complete hyperbolic surfaces (so either closed or finite area with cusps), namely 
\[
N_X(L) \underset{L\to +\infty} \sim\frac{e^L}{2L}.
\]
Note that the $2$ in the denominator just comes from taking unoriented geodesics. More generally, if $X$ has geodesic boundary, then
\[
N_X(L) \underset{L\to +\infty} \sim \frac{e^{h_XL}}{2h_XL}
\]
where $h_X$ is the entropy of $X$. The entropy can mean either the volume entropy or the topological entropy of the geodesic flow on $X$, so to get the above statement for all of these notions, one must combine the work of several authors \cite{Guillope,Manning,Margulis,ParryPollicott,Sullivan}. 

In any case, with these theorems in hand, the entropy of $X$ is the result of the following limit:
\[
h_X:= \lim_{L\to +\infty} \frac{\log\left(N_X(L)\right)}{L}.
\]
Notable properties of $h_X$ include the fact that it is directly related to the Hausdorff dimension of the limit set (of the surface group associated to $X$ as it acts on the hyperbolic plane), and in particular this implies that if $X$ has geodesic boundary, then $h_X<1$, whereas if not, $h_X=1$ (see for instance \cite{Beardon1971}). 

\subsection{The entropy spectrum}
We end this section with the definition of our main object of study, the entropy spectrum. In short, for a hyperbolic surface $X$, it is the collection of entropies of its subsurfaces. Here is more formal definition.

For a hyperbolic surface $X$, we consider its subsurfaces which are considered to be essential (no boundary component is null-homotopic or peripheral in $X$) and geodesic (each boundary component is the geodesic representative of its free homotopy class). Note that, as for curves, essential geodesic subsurfaces are well-defined representatives of their respective isotopy class. Equivalently, one can think of an essential geodesic surface as a codimension-$0$ submanifold with geodesic boundary. If $X$ has cusps, then $Y$ is also allowed to have cusp ends. A proper subsurface of $X$ is a subsurface $Y$ not equal to $X$.

The (marked) entropy spectrum $\E(X)$ is the multiset 
\[
\E(X) := \{h_Y \mid Y \subset X \text{ is a connected subsurface}\}
\]
There is of course an unmarked version as well, where we forget the subsurface associated to the entropies but it is still a multiset so we still know if a certain entropy corresponds to multiple subsurfaces. Note that as a set, it is a subset of $]0,1]$ which contains $1$ if and only if $X$ is geodesically complete and finite area. In this case, the multiplicity of $1$ is exactly $1$, as all proper surfaces are Fuchsian groups of the second kind and as such, as mentioned above, have entropies strictly less than $1$. 

More generally we have the following strict inequality, essentially due to Patterson \cite{Patterson1976}, and which has since been reproven \cite{Furusawa1991} and generalized \cite{DalboOtalPeigne2000, CanaryZhangZimmer2026}.

\begin{lemma}\label{lem:strict}
Let $Y$ be a proper subsurface of a connected surface $X$. Then $h_Y < h_X$. 
\end{lemma}

Note that the connected condition is essential. If $X$ is a disconnected surface with $X_1,\hdots,X_m$ its connected components, then because the number of geodesics of length $L$ is at most $m$ times the maximal growth of its components, $h_X = \max_{i=1,\hdots,m} h_{X_i}$. This, in particular, is why we define the entropy spectrum to be the entropies of connected subsurfaces. If not for this condition, for large enough topology, this would induce infinity multiplicities. The fact that, with our definition, all multiplicities are finite is not obvious, but follows from Theorem \ref{thm:welldone} (see Theorem \ref{thm:order} in the last section). 

We now stated the promised marked rigidity statement. 

\begin{theorem}\label{thm:marked}
The marked entropy spectrum determines a closed hyperbolic surface up to isometry.
\end{theorem}

Our proof uses two lemmas, which are essentially already well-known. 

The first is a result about non-dominating length spectra. Throughout $\T_g$ is the Teichm\"uller space of finite area complete surfaces of fixed topological type. It is slightly stronger than the version usually stated in the literature so we provide a quick proof of how the strengthening works. 

\begin{lemma}\label{lem:dominating}
If $X,Y \in \T_g$ satisfy 
$
 \ell_X(\alpha)\geq \ell_Y(\alpha)
$
for every nonseparating simple closed curve $\alpha$, then $X=Y$.\end{lemma}

\begin{proof}
We first observe that the hypothesis implies the same weak inequality for
separating simple closed curves. Let $\gamma$ be separating. Choose a
nonseparating simple closed curve $\beta$ such that
\[
 i(\beta,\gamma)=2,
\]
and let $T_\gamma$ denote a Dehn twist about $\gamma$. For $n\geq 1$, set
\[
 \beta_n=T_\gamma^n(\beta).
\]
Since a Dehn twist is a homeomorphism, each $\beta_n$ is again
nonseparating. Hence, by assumption,
\[
 \ell_X(\beta_n)\geq \ell_Y(\beta_n).
\]

For either hyperbolic metric $Z\in\{X,Y\}$, the standard Dehn twist
estimate gives
\[
 2n\ell_Z(\gamma)-\ell_Z(\beta)
 <
 \ell_Z(\beta_n)
 \leq
 2n\ell_Z(\gamma)+\ell_Z(\beta);
\]
see, for instance, \cite[Lemma~3.2]{McShaneParlier2008}.

It follows that
\[
 \lim_{n\to\infty}\frac{\ell_Z(\beta_n)}{2n}
 =
 \ell_Z(\gamma).
\]
Dividing the inequality
$\ell_X(\beta_n)\geq\ell_Y(\beta_n)$ by $2n$ and letting $n\to\infty$
therefore gives
\[
 \ell_X(\gamma)\geq\ell_Y(\gamma).
\]
Thus
\[
 \ell_X(\alpha)\geq\ell_Y(\alpha)
\]
for every simple closed curve $\alpha$.

It is now a theorem of Thurston's \cite{ThurstonStretch} that if the above inequality holds then $X=Y$ (see \cite{DoanParlierTan2023} for another proof which uses the Luo-Tan identity \cite{LuoTan2014}). 
\end{proof}

The second result is topological. The following is not difficult via typical curve graph type arguments, and was explicitly shown by Schmutz Schaller \cite{Schmutz-Schaller2000}. Note that on a one-holed torus, this is the connectedness of the curve graph (which is the Farey graph in that case). 

\begin{lemma}\label{lem:sequence}
Let $\alpha$ and $\alpha'$ be two any two non-separating simple closed curves on a finite type surface $\Sigma$ with positive genus. Then there exists a sequence of non-separating closed curves $\alpha_0=\alpha, \alpha_1,\hdots,\alpha_n=\alpha'$ such that 
\[
i(\alpha_k,\alpha_{k-1})=1
\]
for all $k=1,\hdots,n$. 
\end{lemma}

We can now prove the rigidity of the entropy spectrum. 

\begin{proof}[Proof of Theorem \ref{thm:marked}]
Take $X$ and $Y$, points of the same Teichm\"uller space, with the same marked entropy spectrum. By Lemma \ref{lem:dominating} above, if $X \neq Y$, then their exists a non-separating simple closed curve $\alpha$ such that $\ell_X(\alpha) > \ell_Y(\alpha)$. Now consider a separating curve $\beta$ which bounds a pair of pants with $\alpha$ (and hence cuts off a one-holed torus). Let $P_X$ be the geodesic pair of pants on $X$ bounded by $\alpha$ and $\beta$ and $P_Y$ the one on $Y$. By hypothesis, $h_{P_X}= h_{P_Y}$. The entropies only depend on the boundary lengths, and since $\ell_X(\alpha) > \ell_Y(\alpha)$, by monotonicity of the entropies, we have $\ell_X(\beta) < \ell_Y(\beta)$. 

Now take any other non-separating curve $\alpha'$. By Lemma \ref{lem:sequence} above, there exists a sequence of non-separating curves $\alpha_0=\alpha,\hdots,\alpha_n=\alpha'$ such that $i(\alpha_{k},\alpha_{k-1})=1$ for all $k=1,\hdots,n$. Each pair $\alpha_k,\alpha_{k-1}$ lives in a one-holed torus with boundary geodesic $\beta_k$. By the above argument, and arguing by induction, we have
\[
\ell_X(\beta_k) < \ell_Y(\beta_k)
\]
and 
\[
\ell_X(\alpha_k) > \ell_Y(\alpha_k).
\]
In particular, $\ell_X(\alpha') > \ell_Y(\alpha')$ for any nonseparating simple closed curve. By Lemma \ref{lem:dominating}, this is impossible and hence $X=Y$. 
\end{proof}

\begin{remark}
The same proof shows marked entropy rigidity for geodesically complete finite area hyperbolic surfaces (in other words, surfaces with cusps). (Specifically, if the surfaces have positive genus, the proof works verbatim, and for punctured spheres one has to adapt the curve graph argument to a sequence of curves that sequentially intersect twice.) However, the theorem can not hold for {\it any} compact surface with boundary, for a somewhat obvious reason. It is easy to construct two non-isometric pairs of pants, both with geodesic boundary, which have the same entropy, just by using the intermediate value theorem.
\end{remark}

\section{Bounding the number of curves on surfaces with boundary}\label{sec:bound}

The goal of this section is to prove Theorem \ref{thm:mainestimate} and consequences. 

\subsection{Boundary length, width and hexagon decompositions}

For a compact surface $X$ with geodesic boundary, set $B:=\ell_X(\partial X)$ and the width to be $W=w(X)$. 

Note that in the statement of Theorem \ref{thm:mainestimate}, the constants do not depend on the topology of $X$, unlike in Lemma \ref{lem:buser}. This is less surprising when one realizes that in fact, the quantities $B$ and $W$ give bounds on the topology, hence on the area, of $X$.

\begin{proposition} The Euler characteristic $\chi=\chi(X)$ of $X$ satisfies \[ |\chi(X)| < \frac{\sqrt3}{12}\,B\sinh W. \] \end{proposition} \begin{proof} Let $X^D$ denote the double of $X$ across its boundary, and let \[ \Gamma\subset X^D \] be the multicurve obtained from $\partial X$ under the doubling. Then \[ \ell(\Gamma)=B, \qquad \chi(X^D)=2\chi(X). \] Moreover, the covering radius of $\Gamma$ in $X^D$ is exactly $W$.

We now apply an inequality due to Bavard which provides a lower bound on the covering radius of embedded collars \cite{Bavard}. Applying it to $\Gamma \subset X^D$ gives
\[ \sinh W\, \tanh\!\left( \frac{\ell(\Gamma)}{-6\chi(X^D)} \right) > \frac1{\sqrt3}, \]
which in turn implies
\[ \sinh W\, \tanh\!\left( \frac{B}{12|\chi(X)|} \right) > \frac1{\sqrt3}. \] Since $\tanh(x)\leq x$ for all $x>0$, \[ \sinh W\cdot\frac{B}{12|\chi(X)|} > \frac1{\sqrt3}. \] Rearranging yields \[ |\chi(X)| < \frac{\sqrt3}{12}\,B\sinh W. \] \end{proof}

We now consider a hexagon decomposition of $X$, studied in \cite{GultepeParlier, ParlierIsospectral}, consisting in a collection $a_1,\hdots,a_\kappa$ disjoint simple orthogeodesics, where $\kappa$ is the arc complexity of $X$.

The arc complexity is determined by the topology. Indeed, if the decomposition contains $\kappa$ interior arcs and $N$ right-angled hexagons, then each arc is a side of exactly two hexagons, while each hexagon has three arc-type sides. Hence
\[
3N=2\kappa.
\]
Since every right-angled hexagon has area $\pi$, Gauss--Bonnet gives
\[
N\pi=\area(X)=2\pi|\chi(X)|.
\]
Thus
\[
N=2|\chi(X)|
\qquad\text{and}\qquad
\kappa=3|\chi(X)|.
\]
Now each orthogeodesic in a decomposition cuts the boundary $\partial X$ in two points, hence cuts the boundary into $b_1,\hdots,b_{2\kappa}$ subarcs. For simplicity, the lengths of these boundary arcs and the hexagon decomposition arcs will also be denoted by the same letters, and we arrange them by length. 

Thus
\[
 a_1\leq\cdots\leq a_\kappa,
 \qquad
 b_1\leq\cdots\leq b_{2\kappa},
\]
In particular,
\[
 a_\kappa=\max_i a_i,
 \qquad
 b_1=\min_i b_i,
 \qquad
 b_{2\kappa}=\max_i b_i.
\]

Among all possible hexagon decompositions, we choose one where $a_\kappa$ is minimal. We call this a {\it shortest hexagon decomposition}. Note that the set of lengths of orthogeodesics is discrete, so the fact that one exists is immediate. 

We make the following observation.

\begin{proposition}\label{prop:hexagon-bounds}
The extremal lengths of the shortest hexagon decomposition satisfy
\[
2\arcsinh\!\left(\frac{1}{\sinh(B/2)}\right)
\leq a_1\leq a_\kappa\leq 2W
\]
and
\[
2\arcsinh\!\left(\frac{1}{\sinh(2W)}\right)
<b_1\leq b_{2\kappa}<B.
\]
\end{proposition}

\begin{proof}
We follow the strategy of \cite{ParlierIsospectral}, where we construct a hexagon decomposition dual to the boundary of a Voronoi decomposition relative to $\partial X$. More precisely, we put points of $X$ into regions according to which boundary curve they are closest to. (This construction will later be used in Lemma \ref{lem:boundary-reduction} and the cut locus of the decomposition will play an important role in subsection \ref{ss:voronoi}.)

The hexagon decomposition is dual to Voronoi decomposition, hence the distance to the cut locus is bounded above by $W$. As such, the length of the interior arcs of the decomposition are of length at most $2W$. By the minimality of the chosen
hexagon decomposition,
\[
a_\kappa\leq 2W.
\]
Also, every boundary side is a proper subarc of $\partial X$, and hence
\[
b_{2\kappa}<B.
\]

To bound $b_1$ from below, consider a boundary side $b$ of one of the right-angled hexagons, and denote the two adjacent interior sides by $x,y$. It follows that 
\[
b > \arcsinh\!\left(\frac{1}{\sinh x}\right) + \arcsinh\!\left(\frac{1}{\sinh y}\right)
\]
and now, since both $x,y \leq a_\kappa$ and $a_\kappa \leq 2W$, we have
\[
b > 2\arcsinh\!\left(\frac{1}{a_\kappa}\right) \geq 2\arcsinh\!\left(\frac{1}{\sinh(2W)}\right)
\]
and hence the same lower bound holds for $b_1$.

Finally, every interior arc crosses the half-collar associated to the boundary component at each of its two endpoints. Since every boundary component has length at most $B$, the collar lemma gives
\[
a_1
\geq
2\arcsinh\!\left(\frac{1}{\sinh(B/2)}\right).
\]
\end{proof}

\subsection{The main estimate and proof}

We now prove Theorem~\ref{thm:mainestimate}. 

We fix a hexagon decomposition $a_1,\hdots,a_\kappa$, constructed in the previous subsection, which cuts the boundary into subarcs $b_1,\hdots,b_{2\kappa}$, again organized by length. 

By the estimates of the previous subsection, these three quantities are
controlled explicitly in terms of $\ell(\partial X)$ and $w(X)$.

\begin{figure}[ht]
 \centering
\begin{overpic}[width=.5\linewidth,
%grid,
tics=10]{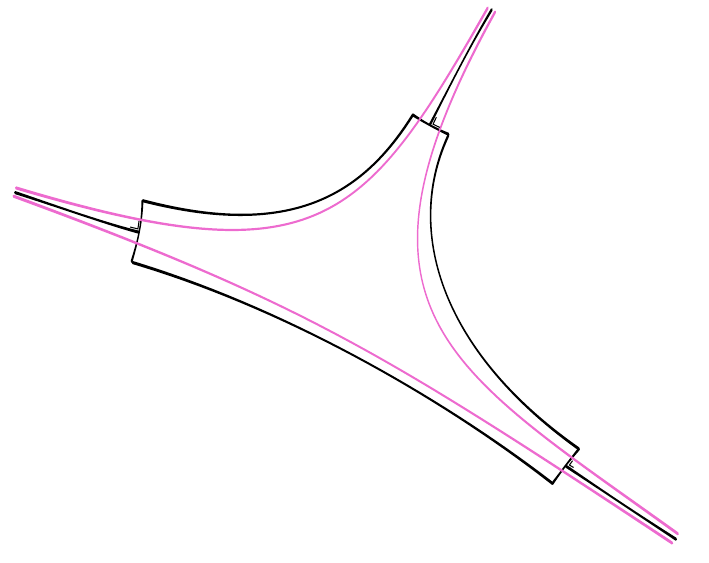}
% \put(32,74){$\alpha$}
% \put(68,21){$\gamma$}
% \put(51,48){\small $x$}
\end{overpic}
\caption{The construction of embedded ideal triangles.}
 \label{fig:Ideal}
\end{figure}

Let $\overline X$ be the complete hyperbolic surface obtained by attaching funnels to the boundary components of $X$. For each boundary side of each right-angled hexagon, consider its geodesic perpendicular bisector. The three bisectors associated to a hexagon define are interior and asymptotic to, and hence define, an embedded ideal triangle in $\overline X$ (see Figure \ref{fig:Ideal}). Let $T\subset\overline X$ be the union of these ideal triangles and put
\[
 S:=X\setminus T.
\]
Thus $S$ is a union of strips separating adjacent ideal triangles (see Figure \ref{fig:Gap02}). 

\begin{figure}[ht]
 \centering
\begin{overpic}[width=.7\linewidth,
%grid,
tics=10]{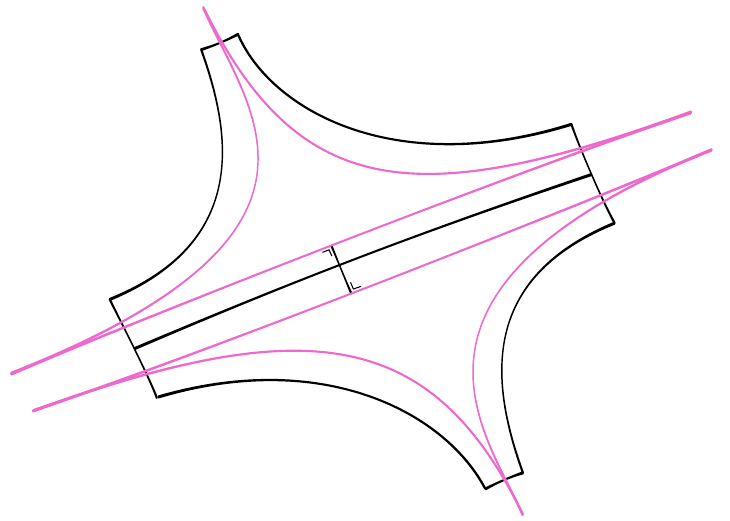}
% \put(32,74){$\alpha$}
% \put(68,21){$\gamma$}
% \put(51,48){\small $x$}
\end{overpic}
\caption{A strip between two ideal triangles}
 \label{fig:Gap02}
\end{figure}

Collapsing these strips identifies the ideal triangles along their sides and produces a
complete finite-area hyperbolic surface $X'$ in which the boundary components
of $X$ have become cusps.
\begin{figure}[h]
 \centering
\begin{overpic}[width=.7\linewidth,
%grid,
tics=10]{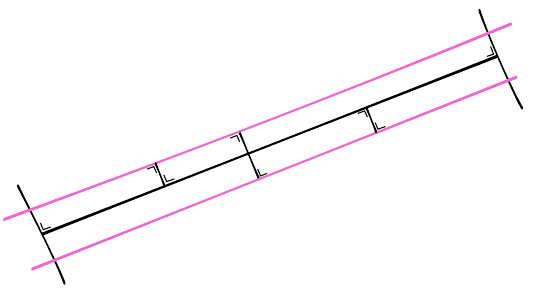}
\put(46,29){$h$}
\put(31,23){$h_1$}
\put(64,28){$h_2$}
\end{overpic}
\caption{The quantities $h$, $h_1$ and $h_2$ from the proof of
Lemma~\ref{lem:strip-width} The pink lines are arcs of an ideal triangle.}
 \label{fig:Gap02zoom}
\end{figure}

We begin by estimating the widths of the strips in terms of our geometric input. 

\begin{lemma}[Strip width]\label{lem:strip-width}
Every strip has width at least
\[
2\arcsinh\!\left(
\frac{1}{
\sinh\!\left(
\frac{a_\kappa}{2}
+
\arcsinh\!\left(\frac{1}{\sinh(b_1/2)}\right)
\right)}
\right).
\]
In particular, its width is always strictly greater than $d_0$ where
\begin{equation}\label{eq:strip-width-lower-bound}
d_0
:=
e^{-a_\kappa/2}\tanh\!\left(\frac{b_1}{2}\right).
\end{equation}
\end{lemma}

\begin{proof}
Let $h$ be the width of one of the strips. As in
Figure~\ref{fig:Gap02zoom}, the geometry on the two sides of the strip
gives two quantities $h_1$ and $h_2$ such that
\[
h\geq h_1+h_2.
\]
We estimate each $h_i$ separately. They are orthogeodesics between a hexagon side $a_j$ and the side of an ideal triangle, so we can estimate them in terms of the local geometry, see Figure \ref{fig:Gap01}.

\begin{figure}[h]
 \centering
\begin{overpic}[width=.9\linewidth,
%grid,
tics=10]{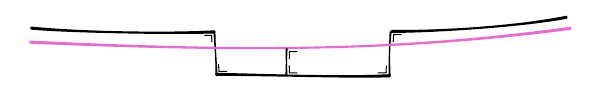}
\put(65.5,7.5){$\frac{b_l}{2}$}
\put(44,6.7){$h_i$}
\put(53,2){$a_j$}
\put(32,7){$\frac{b_k}{2}$}
\end{overpic}
\caption{The local geometry of $h_i$}
 \label{fig:Gap01}
\end{figure}

As we are aiming for a lower bound, we can suppose we are in the extremal situation where both $b_k$ and $b_l$ are as small as possible, that is equal to $b_1$. Furthermore, we can suppose that $a_j$ is as large as possible, hence equal to $a_\kappa$. In this case, by symmetry, $h_i$ splits $a_k$ into two equal parts, and we can now estimate $h_i$ by looking at one of the two sides of the configuration, portrayed in Figure \ref{fig:Gap01zoom}.

\begin{figure}[h]
 \centering
\begin{overpic}[width=.8\linewidth,
%grid,
tics=10]{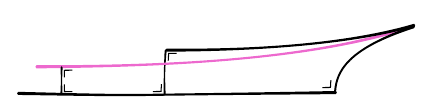}
\put(10.5,4.6){$h_i$}
\put(25,-2){$\frac{a_\kappa}{2}$}
\put(59,-1){$x$}
\put(34,6.5){$\frac{b_1}{2}$}
\end{overpic}
\caption{Estimating $h_1$ and $h_2$}
 \label{fig:Gap01zoom}
\end{figure}

Let $x$ denotes the auxiliary length shown in Figure~\ref{fig:Gap01zoom} which we will use for the computation. The standard identities for an ideal right triangle and a right-angled pentagon give
\[
x
=
\arcsinh\!\left(\frac{1}{\sinh(b_1/2)}\right)
\]
and
\[
\sinh\!\left(\frac{a_\kappa}{2}+x\right)\sinh(h_i)=1.
\]

Consequently,
\[
h_i
=
\arcsinh\!\left(
\frac{1}{
\sinh\!\left(
\frac{a_\kappa}{2}
+
\arcsinh\!\left(\frac{1}{\sinh(b_1/2)}\right)
\right)}
\right).
\]
Now as we only considered the extremal case in our computation, we can conclude that for both $i=1,2$, we have
\[
h_i
\geq
\arcsinh\!\left(
\frac{1}{
\sinh\!\left(
\frac{a_\kappa}{2}
+
\arcsinh\!\left(\frac{1}{\sinh(b_1/2)}\right)
\right)}
\right).
\]
For the simpler bound, set
\[
A=\frac{a_\kappa}{2},
\qquad
u=\frac{b_1}{2}.
\]
Using
\[
\sinh\!\left(\arcsinh\!\left(\frac1{\sinh u}\right)\right)
=\frac1{\sinh u}
\]
and
\[
\cosh\!\left(\arcsinh\!\left(\frac1{\sinh u}\right)\right)
=\coth u,
\]
we have
\[
\frac{1}{
\sinh\!\left(
A+\arcsinh\!\left(\frac1{\sinh u}\right)
\right)}
=
\frac{\sinh u}{\cosh A+\sinh A\cosh u}.
\]
Since
\[
\cosh A+\sinh A\cosh u
\leq e^A\cosh u,
\]
it follows that
\[
\frac{1}{
\sinh\!\left(
A+\arcsinh\!\left(\frac1{\sinh u}\right)
\right)}
\geq
e^{-A}\tanh u.
\]
Thus
\[
h
\geq
2\arcsinh\!\left(e^{-A}\tanh u\right).
\]
Since $0<e^{-A}\tanh u<1$ and $2\arcsinh(t)\geq t$ for $0\leq t\leq1$, we conclude that the strip width satisfies
\[
h
\geq
e^{-a_\kappa/2}\tanh\!\left(\frac{b_1}{2}\right)
=d_0.
\]
\end{proof}

In particular, every arc crossing a strip from one adjacent ideal-triangle to another has length at least $d_0$ where $d_0$ is defined in the previous lemma. 

Recalling that $T$ is the union of the ideal triangles and $S$ is its complement, for a closed geodesic $\gamma\subset X$, write
\[
 \gamma_T:=\gamma\cap T,
 \qquad
 \gamma_S:=\gamma\cap S.
\]
We first show that a definite proportion of the length of $\gamma$ lies in the strips.

\begin{lemma}[Proportion of length in the strips]
\label{lem:strip-proportion}
For every closed geodesic $\gamma$ on $X$,
\[
 \frac{\ell_X(\gamma_S)}{\ell_X(\gamma)}\geq p_0,
\]
where
\begin{equation}\label{eq:strip-proportion-constant}
 p_0
 :=
 \frac{d_0}
 {d_0+\frac32(a_\kappa+b_{2\kappa})}>0.
\end{equation}
\end{lemma}

\begin{proof}
Let $N$ be the number of strip crossings of $\gamma$. By
\eqref{eq:strip-width-lower-bound} of Lemma \ref{lem:strip-width},
\[
 \ell_X(\gamma_S)\geq Nd_0.
\]

Between two consecutive strip crossings, $\gamma$ passes through one of the ideal-triangles contained in a right-angled hexagon. The corresponding segment is no longer than the shorter of the two boundary paths in the hexagon joining its endpoints, and hence no longer than half the perimeter of the hexagon. Since its three interior sides have length at most $a_\kappa$ and its three boundary sides have length at most $b_{2\kappa}$, we obtain
\[
 \ell_X(\gamma_T)
 \leq
 \frac32N(a_\kappa+b_{2\kappa}).
\]
Consequently,
\begin{align*}
 \frac{\ell_X(\gamma_S)}{\ell_X(\gamma)}
 &=
 \frac{\ell_X(\gamma_S)}
 {\ell_X(\gamma_S)+\ell_X(\gamma_T)}\\
 &\geq
 \frac{\ell_X(\gamma_S)}
 {\ell_X(\gamma_S)+\frac32N(a_\kappa+b_{2\kappa})}\\
 &\geq
 \frac{d_0}{d_0+\frac32(a_\kappa+b_{2\kappa})},
\end{align*}
which is the desired estimate.
\end{proof}

We next quantify the length decrease on a single strip. Set
\begin{equation}\label{eq:A0-definition}
 A_0:=a_\kappa+\frac{b_{2\kappa}}2
\end{equation}
and define
\begin{equation}\label{eq:local-strip-contraction}
 r_0
 :=
 \frac{A_0}
 {\arccosh\!\left(\cosh(A_0)\cosh(d_0/2)\right)}.
\end{equation}
Since $d_0>0$, the denominator is strictly larger than $A_0$, and therefore
\[
 0<r_0<1.
\]

\begin{lemma}[Local strip contraction]
\label{lem:local-strip-contraction}
Let $\eta$ be a component of $\gamma_S$, and let $\eta'$ be the corresponding
arc after the strip containing $\eta$ is collapsed. Then
\[
 \ell_{X'}(\eta')\leq r_0\,\ell_X(\eta).
\]
\end{lemma}

\begin{proof}
Write $u=\ell_{X'}(\eta')$. The geometry of the hexagon gives
\[
 u\leq A_0.
\]
Moreover, the strip crossed by $\eta$ has width at least $d_0$. Splitting the
crossing at the perpendicular realizing the strip width and applying the
hyperbolic right-triangle identity gives the comparison
\[
 \ell_X(\eta)
 \geq
 \arccosh\!\left(\cosh(u)\cosh(d_0/2)\right).
\]
Hence
\[
 \frac{\ell_{X'}(\eta')}{\ell_X(\eta)}
 \leq
 \frac{u}
 {\arccosh\!\left(\cosh(u)\cosh(d_0/2)\right)}.
\]
For fixed $t>0$, the function
\[
 u\longmapsto
 \frac{u}{\arccosh(\cosh(u)\cosh t)}
\]
is increasing on $]0,+\infty[$. Since $u\leq A_0$, the right-hand side is at
most $r_0$, proving the claim.
\end{proof}

Combining the two preceding lemmas gives the required uniform contraction.
Define
\begin{equation}\label{eq:global-contraction-constant}
 \lambda
 :=1-(1-r_0)p_0.
\end{equation}
Since $p_0>0$ and $r_0<1$, we have
\[
 0<\lambda<1.
\]
Notice that, using \eqref{eq:strip-width-lower-bound},
\eqref{eq:strip-proportion-constant}, \eqref{eq:A0-definition}, and
\eqref{eq:local-strip-contraction}, this is a completely explicit expression
in $a_\kappa$, $b_1$, and $b_{2\kappa}$.

\begin{proposition}[Uniform length contraction]
\label{prop:uniform-strip-contraction}
For every closed geodesic $\gamma$ on $X$, the geodesic in the corresponding
free homotopy class on the collapsed cusped surface $X'$ satisfies
\[
 \ell_{X'}(\gamma)\leq\lambda\,\ell_X(\gamma).
\]
\end{proposition}

\begin{proof}
The collapse is isometric on the ideal-triangle regions and contracts every
component of $\gamma_S$ by at least the factor $r_0$. Thus the image of
$\gamma$ in $X'$ has length at most
\begin{align*}
 \ell_X(\gamma_T)+r_0\ell_X(\gamma_S)
 &=
 \ell_X(\gamma)-(1-r_0)\ell_X(\gamma_S)\\
 &\leq
 \bigl(1-(1-r_0)p_0\bigr)\ell_X(\gamma)\\
 &=\lambda\ell_X(\gamma),
\end{align*}
where we used Lemma~\ref{lem:strip-proportion}. Geodesic tightening in $X'$
can only decrease length, which proves the proposition.
\end{proof}

We can now finish the counting argument. If $X$ has genus $g$ and $n$
boundary components, then $X'$ has genus $g$ and $n$ cusps. The collapse
preserves free homotopy classes, so Proposition~\ref{prop:uniform-strip-contraction}
gives
\[
 N_X(L)\leq N_{X'}(\lambda L).
\]
Applying Lemma~\ref{lem:buser} to $X'$ yields
\begin{equation}\label{eq:strip-counting-bound}
 N_X(L)
 \leq
 \bigl(g+\lfloor n/2\rfloor\bigr)e^{\lambda L+6}.
\end{equation}
In particular,
\[
 h_X\leq\lambda<1.
\]

Finally, the estimates in the preceding subsection express the extremal
hexagon data $a_\kappa$, $b_1$, and $b_{2\kappa}$, and hence the constants
$d_0$, $p_0$, $r_0$, and $\lambda$, explicitly in terms of the total boundary
length $\ell(\partial X)$ and the width $w(X)$. Together with the corresponding
control of the topology, \eqref{eq:strip-counting-bound} gives the constants
asserted in Theorem~\ref{thm:mainestimate}.
\begin{remark}[Explicit values]\label{rem:explicit}
The contraction constant can be estimated explicitly in terms of the
extremal lengths of the hexagon decomposition. Set
\[
A:=a_\kappa+\frac{b_{2\kappa}}2,
\qquad
C:=a_\kappa+b_{2\kappa},
\qquad
t:=\frac{d_0}{2}
=
\frac12e^{-a_\kappa/2}\tanh\!\left(\frac{b_1}{2}\right).
\]
By the definitions of $p_0$, $r_0$, and $\lambda$,
\[
\lambda
=
1-
\left(
1-
\frac{A}{\arccosh(\cosh A\cosh t)}
\right)
\frac{2t}{2t+\frac32C}.
\]

We now give a simpler, slightly weaker estimate. The elementary inequality
\[
\arccosh(\cosh A\cosh t)\geq\sqrt{A^2+t^2}
\]
gives
\[
1-r_0
\geq
1-\frac{A}{\sqrt{A^2+t^2}}
=
\frac{t^2}{
\sqrt{A^2+t^2}
\bigl(\sqrt{A^2+t^2}+A\bigr)}.
\]
Since
\[
\tanh\!\left(\frac{b_1}{2}\right)
\leq \frac{b_1}{2},
\]
we have
\[
t\leq\frac{b_1}{4}
\leq\frac{b_{2\kappa}}4
\leq\frac C4.
\]
In particular,
\[
A+t
\leq
a_\kappa+\frac34b_{2\kappa}
\leq C,
\]
and hence
\[
1-r_0\geq\frac{t^2}{2C^2}.
\]
Similarly,
\[
p_0
=
\frac{2t}{2t+\frac32C}
\geq
\frac{t}{C}.
\]
It follows that
\[
1-\lambda
=
(1-r_0)p_0
\geq
\frac{t^3}{2C^3}.
\]
Substituting the value of $t$, we obtain
\[
\lambda
\leq
1-
\frac{1}{16}
\left(
\frac{\tanh(b_1/2)}
{a_\kappa+b_{2\kappa}}
\right)^3
e^{-3a_\kappa/2}.
\]

Combining this with Proposition~\ref{prop:hexagon-bounds}, we have
\[
a_\kappa\leq2W,
\qquad
b_{2\kappa}<B,
\qquad
\frac{b_1}{2}
>
\arcsinh\!\left(\frac1{\sinh(2W)}\right)
=
\log\coth W.
\]
Moreover,
\[
\tanh(\log\coth W)=\frac1{\cosh(2W)}.
\]
Therefore
\[
\lambda
<
1-
\frac{e^{-3W}}
{16\,\cosh^3(2W)(B+2W)^3}.
\]
\end{remark}
\subsection{Applications of Theorem \ref{thm:mainestimate}}

\subsubsection{Upper bounding boundary length}

Before proving consequences of Theorem \ref{thm:mainestimate}, we focus on its dependency on $B$ and $W$. It depends on an upper bound for $W$ and on both lower and upper bounds on $B$. Here we show that, at the cost of introducing a dependency on topology, we can remove the dependency on the upper bound on $B$. 

\begin{lemma}[Boundary reduction lemma]
\label{lem:boundary-reduction}
Let $X$ be a compact hyperbolic surface with nonempty geodesic boundary and width $w(X)$. Then there exists a hyperbolic surface $X'$, homeomorphic to $X$, such that
\begin{enumerate}[i.]
 \item
 \[
\ell(\partial X') \leq 4\, \area(X)= 4 \, \area(X')
 \]
 \item
 \[
 w(X')\leq w(X)+\frac{1}{3};
 \]
 \item every closed curve $\gamma$ satisfies \[
 \ell_{X'}(\gamma)\leq \ell_X(\gamma).
 \]
\end{enumerate}

\end{lemma}

\begin{proof}
If $B(X) < 4 \,\area(X)$, we simply take $X'=X$. We therefore assume
that
\[
 B(X)\geq 4 \, \area(X).
\]
Before getting into the proof, here is a brief summary of the strategy. Long boundary forces the existence of short orthogeodesics (and vice-versa). We shall thus modify the surface so that all short orthogeodesics disappear. To do so, we perform a particular kind of strip deformation which consists in removing specific collar neighborhoods of the short curves. 

We consider all orthogeodesics of length at most $\arcsinh(1)$ and call them {\it short}. By the collar estimates in subsection \ref{ss:curves}, they are simple, disjoint and hence can be completed into a hexagon decomposition. 

We now perform a boundary-reducing strip deformation. We consider the collar neighborhood of each short orthogeodesic $a$, of width $\arcsinh(\frac{1}{\sinh{a}})$. 

We consider a somewhat original geometric subset of this collar determined as follows: for an angle $\phi \leq \frac{\pi}{2}$, consider the unique geodesic which makes interior angles $\phi$ with the same two sides as $a$ (see Figure \ref{fig:Boundary01}). 

\begin{figure}[h]
 \centering
\begin{overpic}[width=.5\linewidth,
%grid,
tics=10]{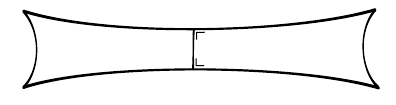}
%\put(25,-2){$\frac{a_\kappa}{2}$}
\put(44,13){$a$}
\put(10,17){$\phi$}
\put(10,8){$\phi$}
\put(88,17){$\phi$}
\put(88,8){$\phi$}
\end{overpic}
\caption{A geometric subset of the collar neighborhood of $a$}
 \label{fig:Boundary01}
\end{figure}
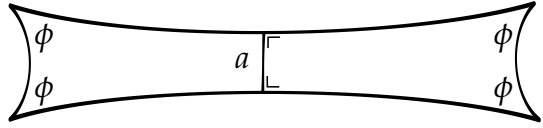

If $\phi = \frac{\pi}{2}$, then this geodesic segment is just $a$. We want to choose $\phi$ so that the segment lies in the collar neighborhood of $a$, and with other properties that will be necessary later. This region will be called the $\phi$-region around $a$. 

We set $\phi = \frac{5 \pi}{12}$ and first observe that the segment is contained in the collar: consider the quadrilateral with side $a$ between two right-angles and the remaining two angles equal to $\phi$. 

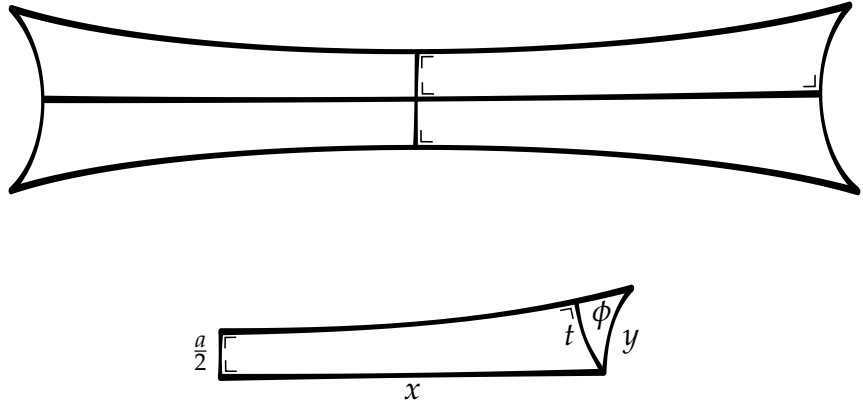
\begin{figure}[h]
 \centering
\begin{overpic}[width=.8\linewidth,
%grid,
tics=10]{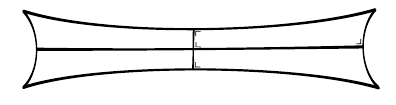}
%\put(25,-2){$\frac{a_\kappa}{2}$}
\end{overpic}

 \vspace{1em}
 
 \begin{overpic}[width=.4\linewidth,
%grid,
tics=10]{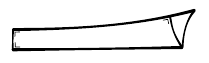}
\put(0,10){$\frac{a}{2}$}
\put(85,18){$\phi$}
\put(45,1.5){$x$}
\put(91.5,13){$y$}
\put(79,13){$t$}

\end{overpic}
\caption{A quadrilateral for computations}
 \label{fig:Boundary02}
\end{figure}

By a standard quadrilateral formula, the side length between a point where the angle is $\phi$ and the side $a$ is of length 
\[
\arcsinh\!\left(\coth\!\left(\frac{a}{2}\right)\cot\!\left(\frac{5\pi}{12}\right)\right). 
\]
And hence we set $F(a)$ to be the difference between the collar width and this quantity:

\[
F(a)
=
\arcsinh\!\left(\frac{1}{\sinh a}\right)
-
\arcsinh\!\left(\coth\!\left(\frac{a}{2}\right) \cot\!\left(\frac{5\pi}{12}\right)\right).
\]
 Elementary calculus shows that this function is strictly decreasing in $a$ and hence has a lower bound when $a=\arcsinh(1)$, so 
\[
F(a) \geq \arcsinh(1)-\arcsinh\!\left((1+\sqrt{2})(2-\sqrt{3})\right) \approx 0.27
\]

We note, for later use, that these neighborhoods contain all points which are very close to $\partial X$. In fact, we can compute a lower bound on the width restricted to all points outside these neighborhoods. Consider the side labelled $t$ in Figure \ref{fig:Boundary02}. Notice that $t$ is the shortest distance to the boundary among all points in the quadrilateral, and hence by symmetry, in the region. We can compute $t$ just by using Lambert quadrilateral formulas. First we compute $x$: 
\[
x=\arcsinh\left(\frac{\cos\phi}{\sinh(a/2)}\right)
\]
and with our choice of $\phi$ we have 
\[
x=\arcsinh\left(\frac{\cos(5\pi/12)}{\sinh(a/2)}\right)
=\arcsinh\left(\frac{\sqrt6-\sqrt2}{4\sinh(a/2)}\right).
\]
Cutting along $t$ produces a smaller Lambert quadrilateral in which we can now compute $t$: 

\[
t
=
\arcsinh\left(
\sinh\left(\frac{a}{2}\right)\cosh(t)
\right)
=
\arcsinh\left(
\sqrt{
\sinh^2\left(\frac{a}{2}\right)
+
\cos^2(\phi)
}
\right)
\]
and hence
\[
t
=
\arcsinh\left(
\sqrt{
\sinh^2\left(\frac{a}{2}\right)
+
\cos^2\left(\frac{5\pi}{12}\right)
}
\right)
=
\arcsinh\left(
\sqrt{
\sinh^2\left(\frac{a}{2}\right)
+
\frac{2-\sqrt{3}}{4}
}
\right).
\] 
Note that, for any $a>0$, we have
\begin{equation}\label{eq:t}
t >
 \arcsinh\left(
\sqrt{
\frac{2-\sqrt{3}}{4} 
} 
\right) > \frac{1}{4}.
\end{equation}
Thus, outside of the $\phi$-regions around the short orthogeodesics, all points with at least two distance paths to boundary, are distance at least the above value from the boundary. 

We now perform a type of strip map, by removing, for each short $a$, the $\phi$-regions from $X$ and pasting the resulting two equal length sides together. Said otherwise, we collapse the $\phi$-regions onto one arc which is the image of the left and the right sides from Figure \ref{fig:Boundary01}. This results in a singular surface $Y$ with cone points along the boundary of interior angle $2\pi -2 \phi$, two for each short orthogeodesic. 

The singular surface $Y$ has an important property: the arcs that make up the singular boundary components of $Y$ all have a minimal length. This is due to the fact that we collapsed regions strictly contained in the collar regions around the short orthogeodesics. As any two short orthogeodesics were at distance at least twice the collar width, the estimate for $F(a)$ shows that each arc of the singular curves are of length twice $F(a)$, hence at least length 
\[
\ell_0:= 2\left( \arcsinh\!\left(\frac{1}{\sinh a}\right)
-
\arcsinh\!\left(\coth\!\left(\frac{a}{2}\right) \cot\!\left(\frac{5\pi}{12}\right)\right)\right).
\]

The singular surface $Y$ has a canonical completion to a hyperbolic surface with smooth geodesic boundary: along every singular boundary component one attaches the unique hyperbolic annulus, singular on its inner side, in which $Y$ embeds isometrically. Denote this completed surface by $X'$. The local picture of singular point of $Y$ and the hyperbolic annulus attached to obtain is portrayed in Figure \ref{fig:Boundary03}.

\begin{figure}[h]
 \centering
\begin{overpic}[width=.5\linewidth,
%grid,
tics=10]{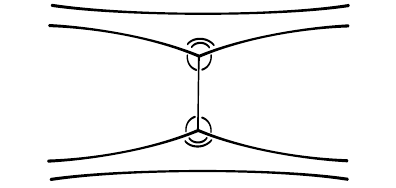}
\put(48,39.5){$2\phi$}
\put(32,30){$\pi - \phi$}
\put(54,30){$\pi - \phi$}
\end{overpic}
\caption{A geometric subset of the collar neighborhood of $a$}
 \label{fig:Boundary03}
\end{figure}

We now proceed to show that $X'$ has the desired properties. First of all, it is not obvious that the above construction works, and in fact, had we not chosen the angle $\phi$ carefully, it wouldn't have. It could be that the completed hyperbolic metric could have a cone point or a cusp instead of the needed geodesic boundary. The fact that we get geodesic boundary follows from our $\ell_0$ lower bound on arc length. The computation is the same one we will use for our next goal: for each boundary component of $\partial X'$, we bound its Hausdorff distance to the corresponding boundary component of $\partial Y$. 

We lift a singular boundary component of $Y$ to $\Hyp$. We consider the resulting broken geodesic (infinite) where each arc is of length at least $\ell_0$ and the angles are all equal to $2 \phi$ (and going in the same direction). We need to show that it has two distinct points at infinity and then compute a bound on the Hausdorff distance between the two geodesics. The extremal case is when all arcs of length exactly $\ell_0$. Supposing the result we aim to show is correct, we can break the extremal situation into a collection of isometric quadrilaterals, with a side of length $\ell_0$ between two angles of $\phi$, and with the other two angles being right angles (see Figure \ref{fig:Boundary04} for a schematic drawing of the situation). It all boils down to whether this picture makes sense with our values. (Said otherwise, whether the two geodesics leaving from the singular points meet in $\Hyp$, on the boundary or not at all.) 

\begin{figure}[h]
 \centering
\begin{overpic}[width=1\linewidth,
%grid,
tics=10]{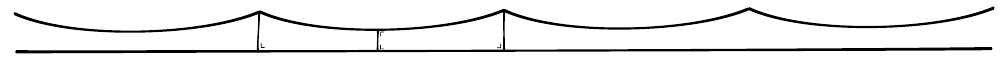}
\put(74,2.7){$2\phi$}
\put(86,5){$\ell_0$}
\end{overpic}
\vspace{0.1cm}
\begin{overpic}[width=0.7\linewidth,
%grid,
tics=10]{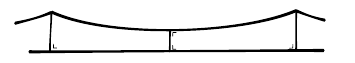}
\put(15,10){$\phi$}
\put(65,6){$\frac{\ell_0}{2}$}
\put(85,7){$w$}
\put(80,10){$\phi$}
\end{overpic}

\caption{Bounding the distance between $Y$ and $X'$}
 \label{fig:Boundary04}
\end{figure}
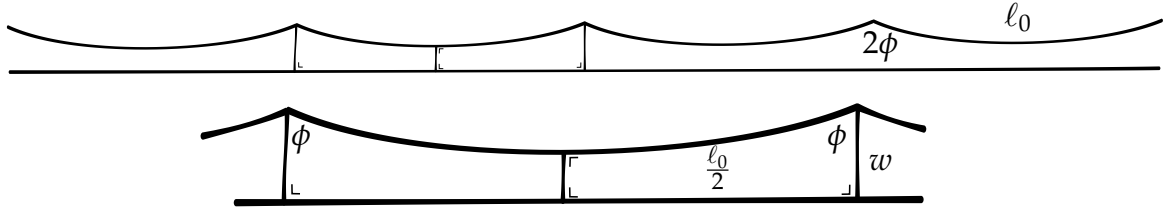

And yet once again, we compute in a Lambert quadrilateral as in Figure \ref{fig:Boundary04}. The Hausdorff distance is bounded by $w$ which satisfies 
\[
\tanh(w) = \cos(\phi) \cosh\left(\frac{\ell_0}{2}\right)
\]
and so we get 

\[ w <
\arctanh\left(
\frac{\sqrt{6}-\sqrt{2}}{4}
\cosh\left(
\arcsinh(1)
-
\arcsinh\left((1+\sqrt{2})(2-\sqrt{3})\right)
\right)
\right) \approx 0.27
\]

and in particular is less than $\frac{1}{3}$. Note that the above computation also shows that the singular geodesics lies in the free homotopy class of a genuine boundary geodesic. Otherwise the computation of $\tanh(w)$ would have resulted in a value $\geq 1$, which would have rendered taking the $\arctanh$ impossible. In particular, we have that
\[
w(X') < w(X) + \frac{1}{3}.
\]
By the standard monotonicity property of strip deformations, removing
the strips weakly decreases the length of every marked closed
geodesic. Therefore, for every nontrivial free homotopy class
$\gamma$,
\[
 \ell_{X'}(\gamma)\leq \ell_X(\gamma).
\]
We next estimate the boundary length. As we removed the $\phi$-regions from $X$ to obtain $Y$, and since for every point of $Y$ the distance to the boundary of $X'$ is even further away, we have that every orthogeodesic on $X'$ is of length at least 

Every arc in the terminal dual decomposition has length at least $2t$ where $t$ was estimated in Equation \ref{eq:t}. Thus the regions of width $t$ around $\partial X'$ are embedded. By the standard area argument, it follows that 
\[
\ell(\partial X') \sinh(t )< \area(X')= \area(X)
\]
and thus 
\[
\ell(\partial X') < \frac{1}{\sinh(t)} \area(X') < \left( \sqrt{6} + \sqrt{2} \right) \area(X') < 4 \,\area(X').
\]
The surface $X'$ has all of the desired properties. 

\end{proof}

\subsubsection{Diameters}\label{ss:diameter}

Diameters of surfaces are interesting, but sometimes difficult to handle, geometric invariants. For instance, there are very few surfaces for which we actually know their (finite) diameters. Nonetheless, if you are presented with an actual surface, concrete bounds on diameter can often be deduced from how the surface is represented. For instance, if the surface is given by a fundamental domain, twice the covering radius of the domain is an immediate bound. 

In this section we show Corollary \ref{cor:diameter-explicit}, an explicit reformulation of Corollary \ref{cor:diam} from the introduction. 

We first show that the diameter of $X$, a compact hyperbolic surface with boundary geodesics, controls both the topology and the total
boundary length. Recall that $\kappa$ is the arc complexity of $X$. 

\begin{lemma}\label{lem:diameter-boundary}
We have
\[
 \kappa\leq 3\bigl(\cosh(\diam(X))-1\bigr)
\]
and
\[
 B\leq 2(\kappa+1)\diam(X)
 \leq 2\diam(X)\bigl(3\cosh(\diam(X))-2\bigr).
\]
\end{lemma}

\begin{proof}
Fix a point $p\in\partial X$. Consider minimizing geodesic paths from $p$ to
points of $\partial X$, where two such paths are considered equivalent if they
are homotopic while the second endpoint is allowed to move along its boundary
component. Apart from the class contained in the boundary component through
$p$, all of these classes are essential arc classes.

Choose a length minimizing representative for each essential arc class. Since
all of these geodesic segments leave from $p$ and they are distance paths, two of them cannot intersect transversely in their interiors: cutting at a first intersection would produce a broken minimizing path, which can be shortened. After separating their common endpoint $p$ by a small perturbation, they therefore give pairwise disjoint, pairwise non-isotopic essential arcs. Such a collection extends to a hexagon decomposition, and hence contains at most $\kappa$ arcs. Thus there are at most $\kappa$ essential classes.

Now fix one of these classes and lift $p$ to a point $\widetilde p$ in the convex universal cover $\widetilde X\subset\Hyp^2$. The class determines a lift $L$ of the boundary component containing its second endpoint. If a point $q\in\partial X$ is joined to $p$ by a minimizing path in this class, then its corresponding lift $\widetilde q\in L$ satisfies
\[
 d_{\Hyp^2}(\widetilde p,\widetilde q)\leq \diam(X).
\]
Hence all such endpoints lie in
\[
 L\cap B_{\Hyp^2}(\widetilde p,\diam(X)),
\]
which is an interval of $L$ of length at most $2\,\diam(X)$. The same conclusion
holds for the boundary class through $p$. Since every point of $\partial X$
is the endpoint of some minimizing path from $p$, and there are at most
$\kappa+1$ classes altogether, we obtain
\[
 B\leq 2(\kappa+1)\diam(X).
\]

It remains to control $\kappa$. Fix $x\in X$ and a lift $\widetilde x$.
Every point of $X$ can be joined to $x$ by a path of length at most
$\diam(X)$, so the metric ball
$B_{\widetilde X}(\widetilde x,\diam(X))$ maps onto $X$. Since
$\widetilde X$ is a convex subset of $\Hyp^2$,
\[
 \area(X)
 \leq
 \area\bigl(B_{\Hyp^2}(\widetilde x,\diam(X))\bigr)
 =2\pi\bigl(\cosh(\diam(X))-1\bigr).
\]
By Gauss--Bonnet and the identity $\kappa=3|\chi(X)|$ proved above,
\[
 \kappa
 =3|\chi(X)|
 \leq 3\bigl(\cosh(\diam(X))-1\bigr).
\]
Substituting this into the preceding boundary estimate gives
\[
 B\leq 2\,\diam(X)\bigl(3\cosh(\diam(X))-2\bigr),
\]
as required.
\end{proof}

We can now make the diameter consequence of Theorem~\ref{thm:mainestimate} completely explicit.

\begin{corollary}\label{cor:diameter-explicit}
Let $X$ be a compact hyperbolic surface with nonempty geodesic boundary. Then
\[
N_X(L)
\leq
\frac{e^6}{2}
\bigl(\cosh(\diam(X))+1\bigr)e^{\eta(X)L},
\]
where
\[
\eta(X)
:=
1-
\frac{e^{-3\,\diam(X)}}
{2^7\diam(X)^3
 \cosh^3(2\,\diam(X))
 \bigl(3\cosh(\diam(X))-1\bigr)^3}.
\]
\end{corollary}

\begin{proof}
We have
\[
W=w(X)\leq\diam(X).
\]
and, by Remark~\ref{rem:explicit}, 
\[
\lambda
<
1-
\frac{e^{-3W}}
{16\,\cosh^3(2W)(B+2W)^3}.
\]
Lemma~\ref{lem:diameter-boundary} gives
\[
B+2W
\leq
2\,\diam(X)\bigl(3\cosh(\diam(X))-1\bigr).
\]
Since $W\leq\diam(X)$, we also have
\[
e^{-3W}\geq e^{-3\diam(X)}
\]
and
\[
\cosh(2W)\leq\cosh(2\,\diam(X)).
\]
Replacing the above values gives us
\[
\lambda
<
1-
\frac{e^{-3\diam(X)}}
{2^7\diam(X)^3
 \cosh^3(2\,\diam(X))
 \bigl(3\cosh(\diam(X))-1\bigr)^3}
=
\eta(X).
\]
Finally, if $X$ has genus $g$ and $n$ boundary components, then
\[
g+\lfloor n/2\rfloor
\leq
\frac{|\chi(X)|+2}{2}
\leq
\frac{\cosh(\diam(X))+1}{2}.
\]
Applying \eqref{eq:strip-counting-bound} and using
$\lambda<\eta(X)$ yields
\[
N_X(L)
\leq
\frac{e^6}{2}
\bigl(\cosh(\diam(X))+1\bigr)e^{\eta(X)L},
\]
which proves the result.
\end{proof}

\subsubsection{Non-filling curves}\label{ss:nf}

Here we prove Theorem \ref{cor:nf}. Recall that for a surface $X$ and $L>0$, $N^{\nf}_X(L)$ is the number of primitive closed geodesics of length at most $L$ that do not fill $X$, and $N^0_X(L)$ is the number of simple closed geodesics of length at most $L$. The set of all simple geodesics on $X$ is denoted $\G^0(X)$. 

We restate it as follows.

\begin{theorem}\label{thm:nf}
Let $\varepsilon>0$ and $g\geq 2$. There exist constants $a=a(g,\varepsilon)<1$ and $A=A(g,\varepsilon)$ such that, for every $X\in \M_g^{\varepsilon}$ and every $L>0$, 
\[
N^{\nf}_X(L) \leq A\, e^{aL}.
\]
\end{theorem}

\begin{proof}
The proof uses Theorem \ref{thm:mainestimate}, Lemma \ref{lem:boundary-reduction} and the fact that there is a polynomial upper bound on $N_X^0(L)$. 

Mirzakhani \cite{Mirzakhani2008} showed that, as $L$ grows, $N_X^0(L)$ is asymptotic to $C_X L^{6g-6}$ where $C_X$ is a proper function on $\M_g$. These results followed previous polynomial upper bounds \cite{BirmanSeries,Rees,Rivin} and the case of once-punctured tori \cite{McShaneRivin}. By Mumford compactness, $\M_g^\varepsilon$ is compact, so, together with the observation that there are at most $3g-3$ simple curves of length at most $2\arcsinh(1)$, there exists a constant $C_{g,\varepsilon} >0$ such that 
\[
N_X^0(L) \leq C_{g,\varepsilon} L^{6g-6}.
\]

Now, by definition, any non-filling closed geodesic lies in the complement of a simple closed geodesic. In addition, observe that if a geodesic $\gamma$ fills $Y$, a proper subsurface of $X$, then $\ell(\partial Y) < 2 \ell(\gamma)$, hence $\gamma$ lies on the complement of a simple curve of length at most $2L$. Hence:
\[
N^{\nf}_X(L) \leq N_X^0( L) + N_X^0( 2 L) \sup_{\alpha \in \G^0(X)} N_{X\setminus \alpha}(L).
\]
The first term covers the simple curves which are, of course, also non-filling. To conclude, we need to show that $N_{X\setminus \alpha}(L)$ is indeed bounded by a constant which depends on $\varepsilon$ and $g$.

For any $\alpha \in \G^0(X)$, the subsurface $Y=X \setminus \alpha$ has width bounded above by the diameter of $X$, hence by a function of $\varepsilon$ and $g$ (as in inequality \ref{eq:systolewidth}). 

Furthermore, the length of $\alpha$ is lower bounded by $\sys(X)\geq \varepsilon$. Now using Lemma \ref{lem:boundary-reduction}, we find an upper bound for all $\alpha$ of length at most $\area(X)$, up to increasing the width by an additive constant. We can now apply Theorem \ref{thm:mainestimate} to show 
\[
N_{X\setminus \alpha}(L) \leq C_1\, e^{C_2 L}
\]
where $C_1,C_2$ are (explicit) functions of $\varepsilon$ and $g$ and which satisfy $C_2 <1$. 
Now in order to conclude, it suffices to incorporate the polynomial term from simple geodesics into the inequality, by choosing any $a > C_2$ (with $a<1$) and increasing $C_1$ to $A$ to account for shorter lengths. 
\end{proof}

Note that, unlike for Theorem \ref{thm:mainestimate}, we do not provide explicit constants because we use a non-explicit inequality for $N_X^0(L)$. However, it is possible to compute such an explicit inequality, but, as this is not the main point of the paper, we omit doing it here. 

As a corollary, we obtain a quantifiable gap in the entropy spectrum. 

\begin{corollary}\label{cor:gapexplicit}
For any $\varepsilon >0$ and $g \geq 2$, there exists a constant $h_{g, \varepsilon} < 1$ such that any $X\in \M_g^\varepsilon$ satisfies
\[
\sup_{Y \subsetneq X} h(Y) \leq h_{g,\varepsilon}.
\]
In particular, there is a quantifiable gap between $1$ and other entropies in $\E(X)$. 
\end{corollary}

\begin{proof}
For every proper subsurface $Y\subset X$ and any $L$, we have 
\[
N_{Y}(L) \leq N^{\nf}_X(L). 
\]
We now apply the bounds from Theorem \ref{thm:nf} to conclude.
\end{proof}

It would be interesting to know whether there are any Laplace spectral interpretations of the above estimates. Recent results concerning eigenvalues of the Laplace spectrum have used the fact that, for a random Weil-Petersson surface \cite{GuthParlierYoung,MirzakhaniRandom}, estimates on the the beginning of the length spectrum (see \cite{AnantharamanMonkI, AnantharamanMonkII, LipnowskiWright, WuXue2025}). In particular, Wu and Xue show that curves up to length $\approx \sqrt{g}$ are mostly simple and, in contrast, by work of Dozier and Sapir, most curves are filling for lengths beyond $\approx (\log g)^2 g$ \cite{DozierSapir}. 

\section{The entropy spectrum}\label{sec:entropy}

We now focus, for closed $X$, on $\E(X)$, the multiset of entropies of subsurfaces of $X$. 

\subsection{Bounding the entropy of surfaces of surfaces with long systoles}

This subsection is dedicated to proving our entropy upper bound for surfaces with long systoles. 

\subsubsection{Voronoi decompositions, carrier graphs and width estimates}\label{ss:voronoi}

Let $X$ be a compact hyperbolic surface with nonempty geodesic boundary, and fix $W>0$ such that
\[
 w(X)\leq W.
\]
We construct a geodesic graph in $X$ which carries every closed curve, with a length distortion depending only on $W$. We state the proposition for closed geodesics but it really holds for any rectifiable non-trivial closed curve. We should think of $G_X$ as equipped with its intrinsic path metric, that is with its edge lengths being the lengths of the corresponding arcs on $X$.

\begin{proposition}[Carrier graph]\label{prop:carrier-graph} There is a finite embedded geodesic graph $G_X\subset X$ such that every closed geodesic $\gamma$ in $X$ is freely homotopic to a closed path $\sigma\subset G_X$ satisfying
\[
 \ell_{G_X}(\sigma)
 \leq \cosh(W)\,\ell_X(\gamma).
\]
\end{proposition}

Before proving the proposition, we construct $G_X$. Let $\widetilde X\subset \Hyp^2$ be the universal cover of $X$, and let $\mathcal L$ be the collection of geodesic lines in $\partial\widetilde X$, that is, the lifts of the boundary components of $X$. For each $L\in\mathcal L$, define
\[
 V_L
 =
 \bigl\{x\in\widetilde X:
 d_{\widetilde X}(x,L)\leq d_{\widetilde X}(x,L')
 \text{ for every }L'\in\mathcal L\bigr\}.
\]
The family $\{V_L\}_{L\in\mathcal L}$ is a locally finite, deck-invariant decomposition of $\widetilde X$. Its quotient in $X$ is the \emph{boundary Voronoi decomposition}.

Equivalently, consider the distance function
\[
 x\longmapsto d_X(x,\partial X).
\]
A minimizing geodesic segment from $x$ to $\partial X$ meets the boundary orthogonally. The cut locus
\[
 G_X
 =
 \bigl\{x\in\operatorname{int}(X):
 x\text{ admits at least two distinct minimizing segments to }\partial X
 \bigr\}
\]
is precisely the interior one-skeleton of the boundary Voronoi decomposition. It is a finite embedded geodesic graph. A point in the interior of an edge of $G_X$ has exactly two minimizing segments to the boundary, whereas a vertex has at least three. Vertices of valence greater than three are allowed (see Figure \ref{fig:Voronoi}). 

\begin{figure}[ht]
\centering
\begin{minipage}{0.45\textwidth}
\centering
\includegraphics[width=\linewidth]{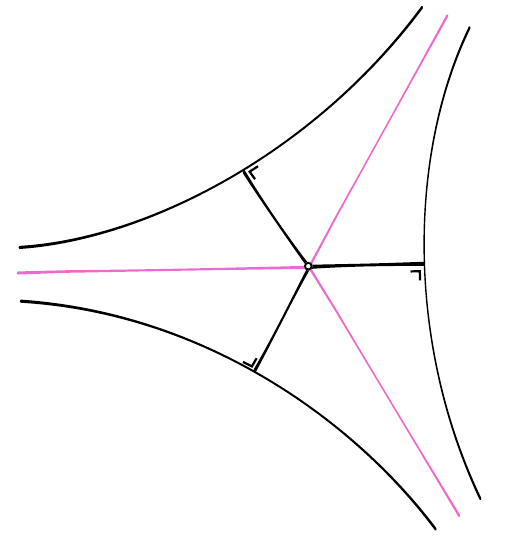}
\end{minipage}
\hfill
\begin{minipage}{0.45\textwidth}
\centering
\includegraphics[width=\linewidth]{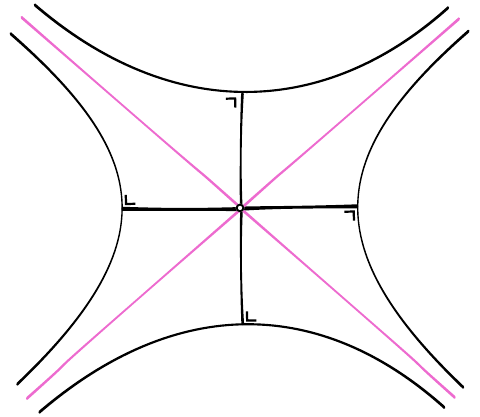}
\end{minipage}
\caption{Local models for the cut locus of the boundary.}\label{fig:Voronoi}
\end{figure}

In Figure \ref{fig:Voronoi}, the exterior segments represent pieces of $\partial X$, the prongs correspond to distance paths from the vertex to $\partial X$ and the pink central set is $G_X$. 

For every vertex of $G_X$, draw all of its minimizing segments to $\partial X$. Together with the edges of $G_X$ and the boundary, these segments cut $X$ into hyperbolic quadrilaterals (see Figure \ref{fig:VoronoiQuad}). More precisely, let $\widetilde e$ be an edge of the lifted Voronoi graph, with endpoints $\widetilde v$ and $\widetilde w$, and let $V_L$ be one of the two Voronoi regions adjacent to $\widetilde e$. If $q_v,q_w\in L$ are the feet of the perpendiculars from $\widetilde v,\widetilde w$ to $L$, then
\[
 \widetilde e,
 \qquad [\widetilde v,q_v],
 \qquad [q_v,q_w]\subset L,
 \qquad [q_w,\widetilde w]
\]
bound one such quadrilateral. Thus every quadrilateral has one side contained in $\partial X$, the opposite side contained in $G_X$, and two lateral sides which are minimizing distance paths. Its two angles on $\partial X$ are right angles, and its lateral sides have length at most $W$.

\begin{figure}[h]
 \centering
\begin{overpic}[width=0.5\linewidth,
%grid,
tics=10]{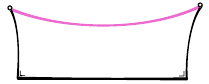}
\put(4,17){$a$}
\put(93,17){$b$}
\put(48,5){$\partial X$}
\put(48,32){$G_X$}
\put(10,28){$\alpha$}
\put(86,27){$\beta$}
\end{overpic}
\caption{A Voronoi quadrilateral.}
 \label{fig:VoronoiQuad}
\end{figure}
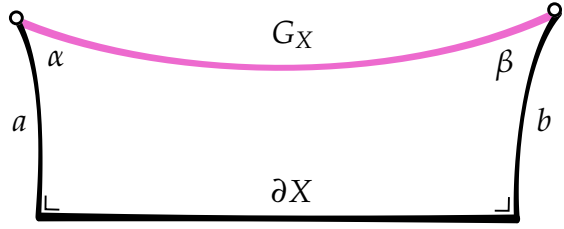

The following estimate controls the two angles adjacent to the side in $G_X$.

\begin{lemma}[Voronoi angle bound]\label{lem:voronoi-angle}
Set
\[
 \theta_0:=\arcsin\!\left(\frac{1}{\cosh W}\right).
\]
Every angle $\alpha$ of a Voronoi quadrilateral adjacent to its side in $G_X$ satisfies
\[
 \theta_0<\alpha<\frac{\pi}{2}.
\]
In particular,
\[
 \frac{1}{\sin\alpha}<\cosh W.
\]
\end{lemma}

\begin{proof}
Consider one of the two vertices of the side in $G_X$, lift it to a point $x\in\widetilde X$, and let
\[
 L_1,\ldots,L_k\in\mathcal L,
 \qquad k\geq 3,
\]
be the boundary lines realizing the distance from $x$ to $\partial\widetilde X$, listed in cyclic order around $x$. Consecutive lines are completed by their common perpendiculars to a right-angled hexagon when $k=3$, and to the analogous right-angled polygon when $k>3$. The cut-locus edges issuing from $x$ lie on the corresponding perpendicular bisectors.

Fix a consecutive pair $L_i,L_{i+1}$ determining the edge under consideration. Let $c_i$ be their common perpendicular, let $m_i$ be its midpoint, let $p_i=c_i\cap L_i$, and let $q_i$ be the foot of the perpendicular from $x$ to $L_i$. The quadrilateral
\[
 x\,m_i\,p_i\,q_i
\]
is a Lambert quadrilateral: its angles at $m_i,p_i,q_i$ are right angles, and its angle at $x$ is precisely the angle $\alpha$ of the corresponding Voronoi quadrilateral. Gauss--Bonnet gives
\[
 \area(xm_ip_iq_i)
 =2\pi-\left(\frac{3\pi}{2}+\alpha\right)
 =\frac{\pi}{2}-\alpha.
\]
Its area is positive, and hence $\alpha<\pi/2$.

For the lower bound, set
\[
 r=d_{\widetilde X}(x,q_i),
 \qquad
 u=d_{\widetilde X}(m_i,p_i).
\]
The standard identities for a Lambert quadrilateral give
\[
 \sin\alpha=\frac{\cosh u}{\cosh r}.
\]
Since $u>0$ and $r\leq W$, it follows that
\[
 \sin\alpha>\frac{1}{\cosh r}
 \geq\frac{1}{\cosh W}.
\]
Geometrically, the limiting equality is the ideal right-triangle configuration. This proves the result.
\end{proof}

We next record two elementary estimates for arcs crossing a Voronoi quadrilateral. The first is illustrated in Figure \ref{fig:VoronoiArc01}.

\begin{lemma}[Crossing a quadrilateral]\label{lem:quadrilateral-crossing} Let $Q$ be a hyperbolic quadrilateral with two consecutive right angles. Let $t_0$ be the side between the right angles and $t'$ the opposite side. Suppose that the remaining side lengths satisfy $a,b\leq W$ and that the two remaining angles lie in $]0,\pi/2[$. Then
\[
 t'\leq \cosh(W)\,t_0.
\]
In particular, if an arc of length $t$ joins the two lateral sides of $Q$, then
\[
 t'\leq \cosh(W)\,t.
\]
\end{lemma}

\begin{proof}
The complete geodesics containing the two lateral sides are perpendicular to the complete geodesic containing the side of length $t_0$. Thus that side is their common perpendicular, and every arc joining the lateral sides has length at least $t_0$.

It remains to estimate $t'$. Using a standard quadrilateral formula (see, e.g., \cite[p.38]{BuserBook}), we have
\begin{equation}\label{eq:carrier-fermi}
 \cosh t'
 =\cosh a\cosh b\cosh t_0-\sinh a\sinh b.
\end{equation}
We maximize $t'$ under the constraints $a,b\leq W$ and that the two remaining angles are nonobtuse. The result will be, by a standard calculus computation, that $t'$ is maximal when $a=b=W$. We provide an argument for completeness.

Put
\[
 C=\cosh t_0,
 \qquad
 u=\tanh a,
 \qquad
 v=\tanh b.
\]
Differentiating \eqref{eq:carrier-fermi} gives
\[
 \sinh(t')\frac{\partial t'}{\partial a}
 =\cosh a\cosh b\,(uC-v),
 \qquad
 \sinh(t')\frac{\partial t'}{\partial b}
 =\cosh a\cosh b\,(vC-u).
\]
By first variation, these derivatives are the cosines of the two remaining angles. The admissible region is therefore characterized by
\[
 uC\geq v,
 \qquad
 vC\geq u,
\]
with strict inequalities for the quadrilaterals occurring in $X$. Hence an interior extremum is impossible. On a boundary component, say $v=Cu$, one has
\[
 \cosh t'
 =C\sqrt{\frac{1-u^2}{1-C^2u^2}},
\]
which is strictly increasing in $u$. Thus an extremum on the closure of the admissible region must have $a=W$ or $b=W$. If, for example, $a=W$, then $\partial t'/\partial b\geq0$ throughout the remaining admissible interval, so the maximum is reached at $b=W$. Therefore the largest possible value of $t'$ is obtained when $a=b=W$.

In this symmetric case, \eqref{eq:carrier-fermi} becomes
\[
 \cosh t'
 =\cosh^2(W)\cosh t_0-\sinh^2(W)
 =1+2\cosh^2(W)\sinh^2\!\left(\frac{t_0}{2}\right).
\]
Therefore
\[
 \sinh\!\left(\frac{t'}{2}\right)
 =\cosh(W)\sinh\!\left(\frac{t_0}{2}\right).
\]
Since $\sinh(\lambda x)\geq\lambda\sinh x$ for $\lambda\geq1$ and $x\geq0$,
we obtain
\[
 t'\leq\cosh(W)t_0\leq\cosh(W)t.
\]
\end{proof}

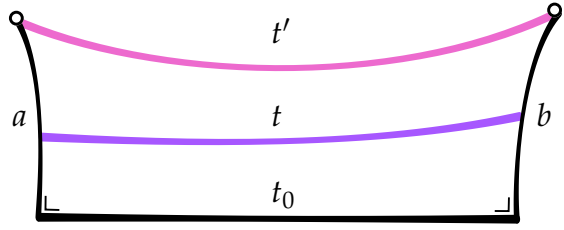
\begin{figure}[h]
 \centering
\begin{overpic}[width=0.5\linewidth,
%grid,
tics=10]{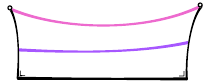}
\put(4,18){$a$}
\put(93,18){$b$}
\put(48,5){$t_0$}
\put(48,32){$t'$}
\put(48,18){$t$}
\end{overpic}
\caption{An arc joining the two lateral sides is replaced by the side in $G_X$.}
 \label{fig:VoronoiArc01}
\end{figure}

The second estimate concerns an arc with one endpoint on a lateral side and the other on the side contained in $G_X$, illustrated in Figure \ref{fig:VoronoiArc02}.

\begin{lemma}[Reaching the carrier graph]\label{lem:quadrilateral-upper-side}
Let $v$ be a vertex of the side of a Voronoi quadrilateral contained in $G_X$, let $q$ lie on that side, and let $p$ lie on the lateral side issuing from $v$. Write
\[
 t'=d_X(v,q),
 \qquad
 t=d_X(p,q),
\]
and let $\alpha$ be the angle at $v$. Then
\[
 \frac{t'}{t}\leq\frac{1}{\sin\alpha}.
\]
In particular, by Lemma~\ref{lem:voronoi-angle},
\[
 t'<\cosh(W)\,t.
\]
\end{lemma}

\begin{proof}
Let $p_0$ be the foot of the perpendicular from $q$ to the complete geodesic containing the lateral side, and set $s=d_X(p_0,q)$. Then $s\leq t$. The right triangle with vertices $v,p_0,q$ has hypotenuse $t'$, side opposite $\alpha$ of length $s$, and
\[
 \sinh(t')\sin\alpha=\sinh(s).
\]
Thus
\[
 \frac{\sinh(t')}{\sinh(s)}=\frac{1}{\sin\alpha}.
\]
Since $x\mapsto\sinh(x)/x$ is increasing on $(0,\infty)$ and $t'\geq s$, we
have
\[
 \frac{t'}{s}\leq\frac{1}{\sin\alpha}.
\]
As $s\leq t$, the desired estimate follows. The final inequality is an immediate consequence of Lemma~\ref{lem:voronoi-angle}.
\end{proof}

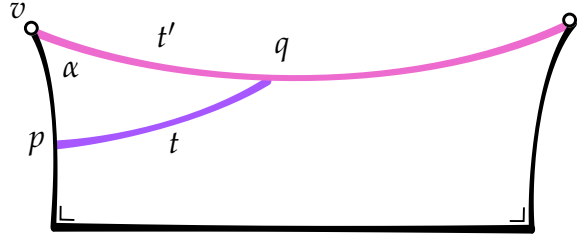
\begin{figure}[h]
 \centering
\begin{overpic}[width=0.5\linewidth,
%grid,
tics=10]{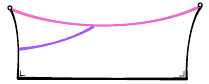}
\put(4,16){$p$}
\put(1,38){$v$}
\put(26,33){$t'$}
\put(46,32){$q$}
\put(28,15){$t$}
\put(10,28){$\alpha$}
\end{overpic}
\caption{An arc from a lateral side to the side in $G_X$ is replaced by a subarc of $G_X$.}
 \label{fig:VoronoiArc02}
\end{figure}

We can now prove the proposition. 

\begin{proof}[Proof of Proposition~\ref{prop:carrier-graph}] Choose a representative of $\gamma$ contained in the interior of $X$ and in general position with respect to the quadrilateral decomposition. Cutting it at its intersections with the sides of the quadrilaterals expresses it as a finite collection of arcs. Each arc is replaced inside its quadrilateral according to the following rules.

\begin{enumerate}[(1)]
\item If its endpoints lie on the two opposite lateral sides, replace it by  the entire side in $G_X$. By Lemma~\ref{lem:quadrilateral-crossing}, this multiplies its length by at most $\cosh W$.

\item If one endpoint lies on a lateral side and the other on the side in $G_X$, collapse the lateral endpoint to the adjacent vertex of $G_X$ and replace the arc by the corresponding subarc of $G_X$. By Lemma~\ref{lem:quadrilateral-upper-side}, this multiplies its length by at most $\cosh W$.

\item If both endpoints lie on the side in $G_X$, replace the arc by the geodesic subarc between them. Its length is no greater than the length of the original arc.

\item If both endpoints lie on the same lateral side, collapse the arc to the  adjacent vertex of $G_X$.
\end{enumerate}

The replacements agree along lateral sides: every point of a lateral side is sent to its endpoint in $G_X$. They therefore concatenate to a closed path $\sigma\subset G_X$. Each replacement is performed inside a single quadrilateral, so $\sigma$ is freely homotopic to $\gamma$. Summing the four local estimates gives
\[
 \ell_{G_X}(\sigma)
 \leq\cosh(W)\,\ell_X(\gamma),
\]
as required.
\end{proof}

\begin{remark}
The strict angle inequality in Lemma~\ref{lem:voronoi-angle} gives a strict local estimate in the second replacement case. For the counting arguments below, the uniform weak bound with constant $\cosh(W)$ is the useful formulation.
\end{remark}

\subsubsection{A combinatorial input: long tangle graphs}

We now isolate the combinatorial estimate that will be applied to the carrier
graphs constructed above. Let $G$ be a finite connected metric graph, meaning with a finite number of edges with length in $\R^{>0}$. Length on $G$, as before, will be denoted $\ell_G$ or simply $\ell$ if the graph at hand is implicit. 

The \emph{girth} of $G$ is the length of its shortest embedded cycle. We define its
\emph{tangle length} by
\[
 \operatorname{tl}(G)
 =
 \inf\bigl\{
 \ell_G(H):
 H\subset G \text{ is connected and } \beta_1(H)\geq 2
 \bigr\},
\]
with the convention that the infimum of the empty set is $+\infty$. Thus, below the tangle length, a connected subgraph contains at most one independent cycle.

The volume entropy of $G$ is
\[
 h_G
 =
 \limsup_{R\to+\infty}
 \frac{1}{R}
 \log\length_{\widetilde G}\bigl(B_{\widetilde G}(x,R)\bigr),
\]
where $\widetilde G$ is the universal cover of $G$, equipped with the lifted path metric $\ell_{\widetilde G}$. As usual, this does not depend on the choice of $x\in\widetilde G$.

\begin{lemma}[Long tangle graphs]\label{lem:long-tangle-graphs} Let $G$ be a finite connected metric graph with $N$ edges. Suppose that its girth is at least $b>0$ and that
\[
 \operatorname{tl}(G)\geq T>0.
\]
Then
\[
 h_G
 \leq
 \frac{4(N+1)}{T}
 \log\left[
 (6N+10)
 \left(
 1+\frac{T}{4(N+1)b}
 \right)
 \right].
\]
Consequently, for any family in which the number of edges is uniformly bounded above and the girth is uniformly bounded below, we have
\[
 h_G
 =
 O\left(\frac{\log(2+T)}{T}\right)
\]
as the tangle-length lower bound $T$ tends to infinity.
\end{lemma}

\begin{proof}
We first record a local estimate. Let $H$ be a connected metric graph with at most $m$ edges, first Betti number at most one, and girth at least $b$. We claim that, for every $y\in\widetilde H$ and every $R\geq0$,
\begin{equation}\label{eq:unicyclic-sphere-growth}
 \#S_{\widetilde H}(y,R)
 \leq
 (3m+4)\left(1+\frac{R}{b}\right),
\end{equation}
where $S_{\widetilde H}(y,R)$ denotes the metric sphere of radius $R$ centered at $y$.

If $H$ is a tree, then $\widetilde H=H$. After subdividing the edge containing $y$, if necessary, the point $y$ is a vertex and the resulting graph has at most $m+1$ edges. The sphere meets each edge in at most one point, and hence has cardinality at most $m+1$.

Suppose now that $\beta_1(H)=1$, and let $C$ be its unique embedded cycle. The full preimage of $C$ in $\widetilde H$ is a geodesic line $A$. The remaining edges form finite trees attached periodically along $A$, with period $\ell(C)\geq b$. Let $a\in A$ be the gate of $y$ to $A$. The component of $\widetilde H\setminus A$ containing $y$ contributes at most $m+1$ points to the sphere. Every other attached tree which meets the sphere is rooted in the segment $B_A(a,R)$, of length $2R$. This segment meets at most
\[
 2+\frac{2R}{\ell(C)}
 \leq
 2+\frac{2R}{b}
\]
fundamental periods. In each period there are at most $m$ edges outside the
axis, and the sphere meets each such edge in at most one point. The axis
itself contributes at most two further points. Therefore
\[
 \#S_{\widetilde H}(y,R)
 \leq
 (m+1)
 +m\left(2+\frac{2R}{b}\right)
 +2
 =
 3m+3+\frac{2mR}{b},
\]
which implies \eqref{eq:unicyclic-sphere-growth}.

We now apply this estimate locally in $\widetilde G$. Set
\[
 r=\frac{T}{4(N+1)}.
\]
Fix $x\in\widetilde G$, and let $p$ be its projection to $G$. Subdivide the edge containing $p$, if necessary, so that $p$ is a vertex. Now the resulting graph has at most $N+1$ edges.

For $0<s\leq r$, the projection of $B_{\widetilde G}(x,s)$ is the metric ball $B_G(p,s)$. Its intersection with each edge of the subdivided graph consists of at most two intervals, adjacent to the endpoints of that edge, each of length at most $s$. It follows that
\[
 \length\bigl(B_G(p,s)\bigr)
 \leq
 2(N+1)s
 \leq
 \frac{T}{2}.
\]
The definition of tangle length therefore gives
\[
 \beta_1\bigl(B_G(p,s)\bigr)\leq1.
\]
After cutting at the endpoints of the intervals above, the graph $B_G(p,s)$ has at most
\[
 2(N+1)
\]
edges, and every cycle in it has length at least $b$.

The inclusion $B_G(p,s)\hookrightarrow G$ is injective on fundamental groups. Consequently, the component of its full preimage in $\widetilde G$ which contains $x$ is its universal cover, and it contains $B_{\widetilde G}(x,s)$. Applying \eqref{eq:unicyclic-sphere-growth} with $m=2(N+1)$ gives
\[
 \#S_{\widetilde G}(x,s)
 \leq
 (6N+10)\left(1+\frac{r}{b}\right)
 =:M
\]
for every $x\in\widetilde G$ and every $0<s\leq r$.

The universal cover $\widetilde G$ is a tree. Hence every point of $S_{\widetilde G}(x,(k+1)r)$ has a unique ancestor on $S_{\widetilde G}(x,kr)$, and each such ancestor has at most $M$ descendants at distance $r$. By induction this means that
\[
 \#S_{\widetilde G}(x,kr)\leq M^k
\]
for every integer $k\geq0$. Allowing a final step of length at most $r$ gives
\[
 \#S_{\widetilde G}(x,t)
 \leq
 M^{\lceil t/r\rceil}
\]
for every $t\geq0$. Integrating the cardinalities of the metric spheres, we obtain
\[
 \length\bigl(B_{\widetilde G}(x,R)\bigr)
 \leq
 R\,M^{\lceil R/r\rceil}.
\]
It follows that
\[
 h_G\leq\frac{\log M}{r}.
\]
Substituting the definitions of $M$ and $r$ proves the stated estimate.
\end{proof}

We note that the volume entropy of graphs was studied by Lim \cite{Lim} and, like in the current paper, used to study entropies of manifolds by Balacheff and Merlin \cite{BalacheffMerlin}. 

\subsubsection{The high systole entropy bound}

We now return to the carrier graph $G_X$ associated to a compact connected hyperbolic surface $X$ with nonempty geodesic boundary. As before, $G_X$ is equipped with the metric induced from $X$. Recall that every vertex has valence at least three. Whenever we consider subgraphs, we will systematically suppress its vertices of valence two. This will not change their entropy, girth, or tangle length.

The first observation relates the tangle length of $G_X$ to the systole of
$X$.

\begin{lemma}[Tangle length and systole]\label{lem:carrier-tangle-systole}
If
\[
 \sys(X)\geq S,
\]
then
\[
 \operatorname{tl}(G_X)\geq \frac{S}{2}.
\]
\end{lemma}

\begin{proof}
Let $H\subset G_X$ be a connected subgraph with $\beta_1(H)\geq2$. Removing hanging trees and then passing to a minimal connected rank-two subgraph gives a core $K\subset H$ such that
\[
 \beta_1(K)=2
 \qquad\text{and}\qquad
 \length(K)\leq\length(H).
\]
(Throughout the proof, $\ell(\cdot)$ is used for $\ell_{G_X}(\cdot)$.)

After suppressing vertices of valence two, $K$ is a figure-eight, a theta graph, or a barbell graph. In each case there are two embedded cycles $\alpha,\beta\subset K$ which freely generate $\pi_1(K)$ and satisfy
\[
 \ell(\alpha)+\ell(\beta)
 \leq 2\,\length(K).
\]
If one of $\alpha$ and $\beta$ is non-peripheral in $X$, then its geodesic representative has length at least $S$, and therefore
\[
 S\leq 2\,\length(K).
\]
Suppose that both are peripheral. A regular neighborhood of their union is a pair of pants. With one relative choice of orientations, the product of the two based loops represents the third boundary component of this pair of pants; with the other choice, it represents a non-peripheral class (a so-called figure-$8$ curve). Thus one of $\alpha\beta$ and $\alpha\beta^{-1}$ is non-peripheral in $X$. After cyclic reduction, its length is at most
\[
 \ell(\alpha)+\ell(\beta)
 \leq 2\,\length(K).
\]
Its geodesic representative consequently has length at least $S$, so again
\[
 S\leq2\,\length(K)\leq2\,\length(H).
\]
Taking the infimum over all such subgraphs $H$ proves the result.
\end{proof}

We shall also need the following elementary information about the size and girth of the carrier graph.

\begin{lemma}[Size and girth of the carrier graph]
\label{lem:carrier-size-girth}
The graph $G_X$ has at most $-3\chi$ edges. Moreover, if every component of $\partial X$ has length at least $b>0$ and $\sys(X)\geq S$, then
\[
 \operatorname{girth}(G_X)\geq\min\{b,S\}.
\]
\end{lemma}

\begin{proof}
Let $V$ and $E$ be the numbers of vertices and edges of $G_X$. Since $G_X$ is a deformation retract of $X$,
\[
 V-E=\chi.
\]
Every vertex has valence at least three, and hence
\[
 2E=\sum_{v\in V(G_X)}\deg(v)\geq3V.
\]
Substituting $V=E+\chi$ gives $E\leq-3\chi$.

Now let $c\subset G_X$ be an embedded cycle. Since $G_X\hookrightarrow X$ is a homotopy equivalence, $c$ represents an essential free homotopy class in $X$. If this class is peripheral, its geodesic representative is a boundary component and has length at least $b$. Otherwise its geodesic representative has length at least $S$. Geodesic tightening does not increase length, so
\[
 \ell_{G_X}(c)\geq\min\{b,S\}.
\]
Taking the infimum over embedded cycles proves the girth estimate.
\end{proof}

Combining these estimates with Lemma~\ref{lem:long-tangle-graphs} gives the
following explicit entropy bound.

\begin{theorem}[High systole entropy bound]
\label{thm:high-systole-entropy}
Let $X$ be a compact connected hyperbolic surface with nonempty geodesic boundary and Euler characteristic $\chi<0$. Suppose that
\[
 w(X)\leq W,
 \qquad
 \sys(X)\geq S>0,
\]
and that every boundary component of $X$ has length at least $b>0$. Then
\[
 h_X
 \leq
 \frac{8(1-3\chi)\cosh W}{S}
 \log\left[
 (10-18\chi)
 \left(
 1+
 \frac{S}{8(1-3\chi)\min\{b,S\}}
 \right)
 \right].
\]
\end{theorem}

\begin{proof}
By Proposition~\ref{prop:carrier-graph}, every free homotopy class represented by a closed geodesic of length at most $L$ on $X$ is represented on $G_X$ by a closed path of length at most $\cosh(W)L$. Since $G_X$ is a deformation retract of $X$, the two spaces have the same free homotopy classes. Using the equivalent closed-path characterization of the entropy of a finite metric graph, we obtain
\[
 h_X\leq\cosh(W)h_{G_X}.
\]

Set
\[
 N=-3\chi,
 \qquad
 a=\min\{b,S\},
 \qquad
 T=\frac{S}{2}.
\]
By Lemmas~\ref{lem:carrier-tangle-systole} and \ref{lem:carrier-size-girth}, the graph $G_X$ has at most $N$ edges, girth at least $a$, and tangle length at least $T$. Lemma~\ref{lem:long-tangle-graphs} therefore gives
\[
 h_{G_X}
 \leq
 \frac{8(N+1)}{S}
 \log\left[
 (6N+10)
 \left(
 1+\frac{S}{8(N+1)a}
 \right)
 \right].
\]
Multiplying by $\cosh(W)$ and substituting $N=-3\chi$ proves
the result.
\end{proof}

We observe that this result has an immediate consequence (which could certainly have been proven in a more direct way, or deduced from existing results in the literature):

\begin{corollary}\label{cor:zero}
For any $X\in \M_g$, there exists an infinite sequence of $Y_i\subset X$ such that 
\[
\lim_{i \to \infty} h_{Y_i}=0
\]
\end{corollary}

\begin{proof}
Any $X$ contains infinitely many pairs of pants. Their boundary curves are of length lower bounded below by $\sys(X)$, and any sequence of distinct $Y_i$ satisfies 
\[
\lim_{i \to \infty} \sys(Y_i)= +\infty.
\]
Moreover, their widths are all bounded by the diameter of $X$. Hence, the result follows from directly from the estimate in Theorem \ref{thm:high-systole-entropy}.
\end{proof}

The other accumulation points in $\E(X)$ will be identified in the sequel. 

\begin{remark}\label{rem:high-systole-asymptotics}
For fixed topological type and fixed $b>0$, Theorem~\ref{thm:high-systole-entropy} gives
\[
 h_X
 =
 O_{\chi,b}\left(
 \cosh(W)\frac{\log(2+S)}{S}
 \right).
\]
In particular, for a family of homeomorphic surfaces whose widths remain bounded and whose systoles tend to infinity, the entropies tend to zero. More generally, the same conclusion holds whenever
\[
 \cosh(W)\frac{\log(2+S)}{S}\longrightarrow0.
\]
If, in addition, the boundary lengths are bounded below by a fixed positive multiple of $S$, then the logarithmic factor in the inequality of Theorem \ref{thm:high-systole-entropy} remains bounded and
\[
 h_X=O_{\chi}\left(\frac{\cosh(W)}{S}\right),
\]
where the implicit constant also depends on that fixed multiple.
\end{remark}

The lower bound on the boundary lengths can itself be expressed in terms of the width. Define
\[
 \omega(t)
 =
 \arcsinh\left(\frac{1}{\sinh t}\right).
\]
The collar lemma gives an embedded half-collar of width
\[
 \omega\left(\frac{\ell_X(\beta)}{2}\right)
\]
around every boundary component $\beta$. The inner half of this half-collar lies in the Voronoi region of $\beta$. Indeed, if
\[
 0\leq r\leq
 \frac12\omega\left(\frac{\ell_X(\beta)}{2}\right),
\]
and $x$ is the point at distance $r$ from $\beta$ along a perpendicular geodesic, then for every other boundary component $\beta'$,
\[
 d_X(x,\beta')
 \geq
 d_X(\beta,\beta')-r
 \geq
 \omega\left(\frac{\ell_X(\beta)}{2}\right)-r
 \geq r.
\]
Thus
\[
 d_X(x,\partial X)=r,
\]
and consequently
\[
 W\geq
 \frac12\omega\left(\frac{\ell_X(\beta)}{2}\right).
\]
Since $\omega$ is decreasing and satisfies $\omega\circ\omega=\operatorname{id}$, it follows that
\[
 \ell_X(\beta)
 \geq
 2\,\omega(2W)
 =
 2\arcsinh\left(\frac{1}{\sinh(2W)}\right).
\]
Set
\[
 a_W(S)
 =
 \min\left\{
 S,
 2\arcsinh\left(\frac{1}{\sinh(2W)}\right)
 \right\}.
\]
Applying Theorem~\ref{thm:high-systole-entropy} with this lower bound yields the following consequence.

\begin{corollary}\label{cor:high-systole-width-entropy}
Let $X$ be a hyperbolic surface with \[
 w(X)\leq W
 \qquad\text{and}\qquad
 \sys(X)\geq S>0.
\]
Then
\[
 h_X
 \leq
 \frac{8(1-3\chi)\cosh W}{S}
 \log\left[
 (10-18\chi)
 \left(
 1+
 \frac{S}{8(1-3\chi)a_W(S)}
 \right)
 \right].
\]
In particular, for fixed $W$ and $\chi$,
\[
 h_X
 =
 O_{W,\chi}\left(\frac{\log(2+S)}{S}\right)
 \qquad\text{as }S\to+\infty.
\]
\end{corollary}

Since
\[
 \omega(2W)\sim2e^{-2W}
 \qquad\text{as }W\to+\infty,
\]
the last estimate has the rough form
\[
 h_X
 =
 O_{\chi}\left(
 \frac{\cosh W}{S}
 \bigl(1+W+\log(2+S)\bigr)
 \right).
\]

\subsection{Converging sequences of entropies}

We now show that attaching pieces with increasingly large systole does not change the entropy in the limit, provided their topology and width remain controlled. More precisely, let $Y$ be a fixed compact hyperbolic surface with geodesic boundary, not necessarily connected, and let
\[
 \C=\beta_1\sqcup\cdots\sqcup\beta_q\subset\partial Y
\]
be a fixed nonempty union of boundary components. For each $i$, let $Z_i$ be a compact hyperbolic surface of a fixed topological type, again not necessarily connected, with $q$ distinguished boundary components isometric to the curves
$\beta_j$, and let
\[
 Y_i=Y\cup_{\C} Z_i
\]
be the surface obtained by gluing the corresponding boundary components. We write
\[
 S_i=\sys(Z_i)
\]
If $Z_i$ is disconnected, then its systole is the minimum of the systoles of its connected components. As above, the entropy of a disconnected surface is the maximum of the entropies of its connected components.

\begin{theorem}[Entropy convergence under high-systole attachments]
\label{thm:entropy-convergence}
Suppose that there is a constant $W>0$ such that
\[
 w(Z_i)\leq W
\]
for every $i$, and that $
 S_i\longrightarrow+\infty.$ Then $
 h_{Y_i}\longrightarrow h_Y.$
\end{theorem}

The lower bound is immediate as every closed geodesic contained in $Y$ remains a closed geodesic of the same length in $Y_i$ (and in fact is strict, see Lemma \ref{lem:strict}). The rest of the section is devoted to the reverse inequality.

We first extract from the preceding subsection the quantitative counting statement that will be needed in the proof. Recall that an essential orthogeodesic in $Z_i$ whose endpoints lie on $\C$ is not homotopic, with endpoints allowed to move on $\C$, to an arc of $\C$.

\begin{lemma}\label{lem:quantitative-high-systole-counting}
There are sequences $\varepsilon_i\longrightarrow0$, $A_i\geq1$ such that
\[
 \log A_i=o(S_i)
\]
and, for every integer $n\geq0$, the number of curves on $Z_i$ whose lengths belong to $[n,n+1[$ is at most
\begin{equation}\label{eq:Zi-curve-counting}
 A_i e^{\varepsilon_i n}.
\end{equation}

Moreover, if $\mathcal O_i$ denotes the set of essential orthogeodesics in $Z_i$ with endpoints on $\C$, then every $\eta\in\mathcal O_i$ satisfies
\begin{equation}\label{eq:orthogeodesic-lower-bound}
 \ell_{Z_i}(\eta)\geq R_i:=\frac{S_i}{2}-B,
 \qquad
 B_{\max}:=\max_{1\leq j\leq q}\ell_Y(\beta_j),
\end{equation}
and there are sequences $ \delta_i\longrightarrow0$ and $K_i\geq1$ with $\log K_i=o(R_i)$ such that
\begin{equation}\label{eq:Zi-orthogeodesic-counting}
 \#\bigl\{\eta\in\mathcal O_i:
 n\leq\ell_{Z_i}(\eta)<n+1\bigr\}
 \leq K_i e^{\delta_i n}
\end{equation}
for every integer $n\geq0$.
\end{lemma}

\begin{proof}
Since the topological type of $Z_i$ is fixed, the number of its connected components is uniformly bounded. Applying the argument below to each component separately and summing the resulting estimates only changes the multiplicative constants. We may therefore argue as if $Z_i$ were connected.

We briefly explain why the estimates proved above give the slightly stronger, uniform form \eqref{eq:Zi-curve-counting}. Since the topology of $Z_i$ is fixed, the corresponding carrier graphs have a uniformly bounded number of edges. The width bound and the collar estimate used in Corollary~\ref{cor:high-systole-width-entropy} give a uniform positive lower bound
\[
b_W=2\arcsinh\left(\frac{1}{\sinh(2W)}\right)
\]
for the lengths of all boundary components of $Z_i$. For all sufficiently
large $i$, Lemmas~\ref{lem:carrier-tangle-systole} and
\ref{lem:carrier-size-girth} therefore give a uniform edge bound, girth at least $b_W$, and tangle length at least $S_i/2$.

In the proof of Lemma~\ref{lem:long-tangle-graphs}, one may consequently take
\[
 r_i\asymp S_i,
 \qquad
 M_i=O(1+S_i),
\]
where $\asymp$ means comparable up to positive multiplicative constants)and obtain, uniformly in the center $x$ of the universal cover of the carrier graph,
\[
 \#S(x,t)\leq M_i^{\lceil t/r_i\rceil}.
\]
The orbit of a vertex is $b_W$-separated, and the carrier graphs have a uniformly bounded number of vertices. It follows that the number of free homotopy classes represented on the carrier graph by closed paths of length at most $t$ is bounded by
\[
 C(1+t)M_i^{1+t/r_i},
\]
where $C$ is independent of $i$. Proposition~\ref{prop:carrier-graph} replaces a curve of length at most $t$ on $Z_i$ by a path of length at most $\cosh(W)t$ on its carrier graph. Since
\[
 \frac{\log M_i}{r_i}
 =O\left(\frac{\log(2+S_i)}{S_i}\right)
 \longrightarrow0,
\]
and since the factor $1+t$ can be absorbed by increasing the exponent by $O(1/r_i)$, we obtain \eqref{eq:Zi-curve-counting} with $\varepsilon_i\to0$ and
\[
 \log A_i=O(\log(2+S_i))=o(S_i).
\]
Changing the finitely many initial constants does not affect these properties.

We now turn to orthogeodesics. Let $\eta\in\mathcal O_i$ have endpoints on $\beta$ and $\beta'$, which are allowed to coincide. By taking two copies of $\eta$ and arcs on $\beta\cup\beta'$ of total length at most $2B$, one obtains a closed curve $\widehat\eta$ of length at most
\begin{equation}\label{eq:closing-orthogeodesic}
 \ell_{Z_i}(\widehat\eta)
 \leq2\,\ell_{Z_i}(\eta)+2B_{\max}.
\end{equation}
There are two possible relative choices of orientation for the boundary arcs. If one of the resulting curves is peripheral, the other is non-peripheral; this is the same elementary pair-of-pants observation used in the proof of Lemma~\ref{lem:carrier-tangle-systole}. We choose the non-peripheral one. Since its geodesic representative has length at least $S_i$, \eqref{eq:closing-orthogeodesic} gives \eqref{eq:orthogeodesic-lower-bound}.

It remains to count the orthogeodesics directly. Let $G_i$ be the carrier graph of $Z_i$. For each distinguished boundary component $\beta_j$, choose a lateral side of the Voronoi decomposition joining a point $x_{i,j}\in\beta_j$ to a point $v_{i,j}\in G_i$. Its length is at most $W$.

Let $\eta$ be an orthogeodesic from $\beta_j$ to $\beta_k$. Join its endpoints to $x_{i,j}$ and $x_{i,k}$ by shortest arcs along the corresponding boundary components, and then join these points to $v_{i,j}$ and $v_{i,k}$ along the chosen lateral sides. Since the distinguished boundary components have lengths at most $B_{\max}$, this gives a path from $v_{i,j}$ to $v_{i,k}$ of length at most
\[
 \ell_{Z_i}(\eta)+B_{\max}+2W.
\]
Distinct relative homotopy classes of orthogeodesics give distinct fixed-endpoint homotopy classes of such paths.

The proof of Proposition~\ref{prop:carrier-graph} applies verbatim to paths whose endpoints lie on the carrier graph, with the endpoints kept fixed. Thus every such class is represented on $G_i$ by a path of length at most
\[
 \cosh(W)\bigl(\ell_{Z_i}(\eta)+B_{\max}+2W\bigr).
\]
Consequently, if
\[
 D:=B_{\max}+2W+1,
 \qquad
 t_n:=\cosh(W)(n+D),
\]
then the number of orthogeodesics with length in $[n,n+1[$ is bounded by the number of fixed-endpoint path classes on $G_i$, between one of the finitely many ordered pairs $(v_{i,j},v_{i,k})$, having length at most $t_n$.

Lift $v_{i,j}$ and $v_{i,k}$ to the universal cover $\widetilde G_i$. Fixed-endpoint path classes from $v_{i,j}$ to $v_{i,k}$ are parametrized by deck translates of the chosen lift of $v_{i,k}$. Hence the same universal-cover estimate used above gives
\[
 \#\{\text{such path classes of length at most }t\}
 \leq C(1+t)M_i^{1+t/r_i},
\]
where $C$ is independent of $i$ and of the ordered pair of endpoints. Therefore
\[
 \#\bigl\{\eta\in\mathcal O_i:
 n\leq\ell_{Z_i}(\eta)<n+1\bigr\}
 \leq
 Cq^2(1+t_n)M_i^{1+t_n/r_i}.
\]

For all sufficiently large $i$,
\[
 1+t_n\leq(1+r_i)e^{t_n/r_i}.
\]
Thus \eqref{eq:Zi-orthogeodesic-counting} holds with
\[
 \delta_i
 =
 \frac{\cosh(W)(1+\log M_i)}{r_i}
\]
and, for instance,
\[
 K_i
 =
 Cq^2(1+r_i)M_i
 \exp\left(
 \frac{\cosh(W)D(1+\log M_i)}{r_i}
 \right).
\]
Since
\[
 r_i\asymp S_i,
 \qquad
 M_i=O(1+S_i),
\]
we have
\[
 \delta_i\longrightarrow0
\]
and
\[
 \log K_i=O(\log(2+S_i))=o(S_i)=o(R_i).
\]
Changing the finitely many initial constants completes the proof.
\end{proof}

We shall also use a fixed counting estimate for arcs in $Y$. Choose once and for all disjoint collars of the curves $\beta_j$ and a marked point on each boundary component of the truncated surface. In each connected component of $Y$, choose a basepoint and paths joining the marked boundary points in that component to its basepoint. After recording separately the displacement and winding inside the collars, each piece of a curve lying in $Y$ determines a relative homotopy class of arc between two marked boundary points. We call these the \emph{normalized arc classes} of $Y$.

\begin{lemma}[Counting arcs in the fixed piece]
\label{lem:fixed-piece-arc-counting}
For every $a>h_Y$, there is a constant $C_Y(a)$ such that, for every integer $n\geq0$, the number of normalized arc classes in $Y$ whose geodesic representatives have length in $[n,n+1[$ is at most
\begin{equation}\label{eq:Y-arc-counting}
 C_Y(a)e^{an}.
\end{equation}
\end{lemma}

\begin{proof}
Write
\[
Y=Y^{(1)}\sqcup\cdots\sqcup Y^{(r)}.
\]
For an arc contained in $Y^{(j)}$, closing it to the chosen basepoint of $Y^{(j)}$ changes its length by a uniformly bounded amount and has uniformly bounded multiplicity. Since
\[
h_{Y^{(j)}}\leq h_Y<a,
\]
the standard equality between entropy and based orbit-growth gives the required estimate on each component. Summing over the finitely many components gives \eqref{eq:Y-arc-counting}.
\end{proof}

\begin{proof}[Proof of Theorem~\ref{thm:entropy-convergence}]
Fix
\[
 s>h_Y
\]
and choose $a$ so that
\[
 h_Y<a<s.
\]
We shall show that $h_{Y_i}\leq s$ for all sufficiently large $i$.

Let $\gamma$ be a closed geodesic in $Y_i$ which meets both $Y$ and $Z_i$. Cutting $\gamma$ at its intersections with $\C$, replacing each resulting piece by the shortest representative in its relative homotopy class, and normalizing its endpoints in the fixed collars gives a cyclic code
\[
 a_1c_1^-\eta_1c_1^+\,
 a_2c_2^-\eta_2c_2^+\cdots
 a_mc_m^-\eta_mc_m^+.
\]
Here $a_j$ is a normalized arc in $Y$, $\eta_j\in\mathcal O_i$ is an essential orthogeodesic in $Z_i$, and $c_j^-,c_j^+$ record the boundary displacements and winding in the collars. Excursions homotopic into $\partial Z_i$ are absorbed into the collar terms. The resulting cyclic code determines the free homotopy class of $\gamma$; overcounting such codes will be harmless.

There is a constant $D\geq1$, depending only on the fixed topological data, which bounds the number of choices of boundary labels and orientations per block. There is also a constant $K\geq0$, depending only on the chosen collars, such that
\begin{equation}\label{eq:entropy-convergence-length-bookkeeping}
 \sum_{j=1}^m\ell_Y(a_j)
 +\sum_{j=1}^m\ell_{Z_i}(\eta_j)
 +\sum_{j=1}^m\bigl(\ell(c_j^-)+\ell(c_j^+)\bigr)
 \leq \ell_{Y_i}(\gamma)+Km.
\end{equation}
A collar piece is determined by a boundary label, an orientation, and an integer winding number. Since the curves $\beta_j$ are fixed, there is a constant $C_{\rm cyl}$ such that the number of collar pieces with length in $[n,n+1[$ is at most $C_{\rm cyl}$, independently of $i$ and $n$.

Let $M_{i,m}(L)$ denote the number of free homotopy classes represented by closed geodesics of length at most $L$ which make exactly $m$ essential excursions into $Z_i$. Replacing the length of each piece by the integer indexing the unit interval which contains it, and increasing $K$ ifnecessary, Lemmas~\ref{lem:quantitative-high-systole-counting} and \ref{lem:fixed-piece-arc-counting} give
\begin{align*}
 M_{i,m}(L)
 \leq{}&D^m
 \sum_{\substack{u_j,v_j,w_j^-,w_j^+\geq0\\
 \sum_j(u_j+v_j+w_j^-+w_j^+)\leq L+Km}}
 \prod_{j=1}^m
 \left(
 C_Y(a)e^{au_j}
 K_i e^{\delta_i v_j}
 C_{\rm cyl}^2
 \right),
\end{align*}
where the summand is understood to vanish unless
\[
 v_j\geq\lfloor R_i\rfloor-1.
\]

For nonnegative integers with sum $N$, the elementary inequality
\[
 \boldsymbol{1}_{\{N\leq L+Km\}}
 \leq e^{s(L+Km-N)}
\]
separates the sums. We obtain
\begin{equation}\label{eq:fixed-excursion-count}
 M_{i,m}(L)\leq e^{sL}\rho_i(s)^m,
\end{equation}
where
\begin{align*}
 \rho_i(s)
 ={}&e^{sK}D C_Y(a)C_{\rm cyl}^2K_i
 \left(\sum_{u\geq0}e^{-(s-a)u}\right)
 \left(\sum_{v\geq\lfloor R_i\rfloor-1}
 e^{-(s-\delta_i)v}\right)
 \left(\sum_{w\geq0}e^{-sw}\right)^2.
\end{align*}
All three sums are geometric. Since $\delta_i\to0$, we have
$s-\delta_i\geq s/2$ for all sufficiently large $i$, and hence
\[
 \log\rho_i(s)
 \leq
 \log K_i-(s-\delta_i)(R_i-2)+O_s(1).
\]
By Lemma~\ref{lem:quantitative-high-systole-counting},
$\log K_i=o(R_i)$ and $R_i\to+\infty$. Thus
\[
 \rho_i(s)\longrightarrow0.
\]
In particular, $\rho_i(s)<1/2$ for all sufficiently large $i$. Summing
\eqref{eq:fixed-excursion-count} over $m\geq1$ then gives
\begin{equation}\label{eq:mixed-curve-count}
 \sum_{m\geq1}M_{i,m}(L)
 \leq e^{sL}\sum_{m\geq1}\rho_i(s)^m
 \leq e^{sL}.
\end{equation}
Thus the curves which meet both $Y$ and $Z_i$ have exponential growth rate at
most $s$.

Curves contained entirely in $Y$ have exponential growth rate $h_Y<s$. By \eqref{eq:Zi-curve-counting}, curves contained entirely in $Z_i$ have exponential growth rate at most $\varepsilon_i<s$ for all sufficiently large $i$. Finally, the gluing locus $\C$ contributes only the finitely many curves $\beta_1,\ldots,\beta_q$. Together with \eqref{eq:mixed-curve-count}, this shows that, for all sufficiently large $i$, there is a constant $C_i(s)$ such that
\[
 N_{Y_i}(L)\leq C_i(s)e^{sL}
\]
for every $L\geq0$. Therefore
\[
 h_{Y_i}\leq s
\]
for all sufficiently large $i$. Since $s>h_Y$ was arbitrary,
\[
 \limsup_{i\to\infty}h_{Y_i}\leq h_Y.
\]
Combining this with the obvious lower bound proves the theorem.
\end{proof}

\subsection{Identifying accumulation points and well-orderedness}

\begin{figure}[h]
\centering
\begin{overpic}[width=0.49\linewidth,
%grid,
tics=10]{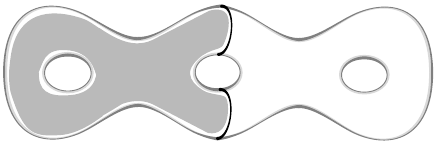}
\end{overpic}
\hfill
\begin{overpic}[width=0.49\linewidth,
%grid,
tics=10]{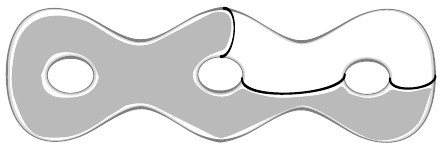}
\end{overpic}
\caption{The entropy of the shaded subsurface on the left is an accumulation point, whereas the entropy of the right is not (unless another smaller subsurface has the same entropy)}
 \label{fig:Entropy}
\end{figure}

In this final section, we prove Theorem \ref{thm:welldone}, which identifies accumulation points in the entropy spectrum and proves well-orderedness. The statements will be proved separately, but the main part is in the first proof.

A consequence of the following theorem is illustrated in Figure \ref{fig:Entropy}.

\begin{theorem}\label{thm:accumulation}
If $h\in ]0,1]$ is an accumulation point of $\E(X)$, then there exists a connected subsurface $Y\subset X$ such that $h_Y=h$ and $X\setminus Y$ has curve complexity at least $1$. Conversely, if $Y\subset X$ is a connected subsurface such that $X\setminus Y$ has curve complexity at least $1$, then $h_Y$ is an accumulation point of $\E(X)$.
\end{theorem}

\begin{proof}

Let $h_{Y_i}$ be a sequence of entropies that converge to a value $h \in [0,1]$ with distinct $Y_i$. Among the $Y_i$, we restrict to a subsequence of homeomorphic $Y_i$. Among this sequence we consider a maximal subsurface $Y$ (not necessarily connected) such that there exists an infinite subsequence with $Y_i=Y\cup Z_i$. Maximality should be thought of in the following sense: if it exists, find a curve which lies in infinitely many $Y_i$s (and consider all surfaces $Y_i$ containing it). Now if there is an infinite subsequence which contains another curve, then restrict to the subsequence of surfaces which contain this new curve. Note that this new subsequence contains the subsurface $Y$ spanned by the two curves. If there is a curve $\alpha$ which does not lie entirely in $Y$ and which belongs to an infinite subsequence, restrict to this subsequence and add $\alpha$ to $Y$ and so on. This process ends in a finite number of steps as the number of curves you can add is bounded by the topology. 

One can think of the resulting sequence as equipped with a sequence of homeomorphisms to a base surface $Y_0=Y \cup Z_0$. Again, among these, there must be an infinite sequence such that the implicit homeomorphism $\varphi_i$ sends $Y_i$ to $Y_0$, $Y$ to $Y$ (fixing each boundary component) and $Z_i$ to $Z_0$. Now there are essentially three cases, but the first two work the same way:
\begin{enumerate}[I.]
\item\label{caseI} $Y$ is empty.
\item\label{caseII} $Y$ is a collection of disjoint simple curves.
\item\label{caseIII} $Y$ contains at least one pair of pants.
\end{enumerate}
The essential observation is that in all cases,
\[
\sys(Z_i)\longrightarrow\infty.
\]
Indeed, otherwise there would exist $L>0$ and an infinite subsequence for which each $Z_i$ contains a non-peripheral closed geodesic of length at most $L$. Since $X$ has only finitely many closed geodesics of length at most $L$, after passing to a further subsequence one of these curves would belong to every $Z_i$. It could then be added to $Y$, contradicting the maximality of $Y$.

In particular, this means that $X\setminus Y$ is of curve complexity at least $1$. 

In Case \ref{caseI}, the high-systole entropy estimate gives
\[
h_{Y_i}\longrightarrow 0,
\]
which is impossible since $h>0$.

In Case \ref{caseII}, the same counting argument as in the proof of Theorem~\ref{thm:entropy-convergence} applies, with the fixed piece replaced by the finite union of persistent closed geodesics. The only contributions along these curves are boundary displacement and winding, and hence have zero exponential growth. Since every essential excursion into $Z_i$ has length tending to infinity, the corresponding geometric-series estimate gives
\[
h_{Y_i}\longrightarrow0.
\]
Thus this case is also impossible.

Finally, in Case \ref{caseIII}, Theorem~\ref{thm:entropy-convergence} applies
and gives
\[
h_{Y_i}\longrightarrow h_Y.
\]
If $Y$ is disconnected, choose one of its maximal entropy connected components. 

Now observe that because there is an infinite sequence of $Z_i$, this means that $X\setminus Y$ is of curve complexity at least $1$. Finally, taking any $Y$ with $X\setminus Y$ of curve complexity at least $1$, consider any $Y_0 = Y \cup Z_0$ obtained by pasting a pair of pants to $Y$ which lies in a connected component of $X\setminus Y$ of curve complexity at least $1$, say $\tilde{Y}$. Now consider a non-peripheral essential simple curve $\alpha$ in $\tilde{Y}$ which essentially intersects $Z_0$. Now consider $Z_i$ the result of $i$ (positive) Dehn twists of $Z_0$ along $\alpha$ and $Y_i$ obtained by pasting $Y$ to $Z_i$. The hypotheses of Theorem~\ref{thm:entropy-convergence} are satisfied. Indeed, the widths of the complementary pieces are uniformly bounded by $\diam(X)$, while their systoles tend to infinity by the same finiteness argument used above: otherwise a bounded-length curve would occur in infinitely many of them, contradicting the maximality of the persistent subsurface. We can conclude that $h_Y$ is an accumulation point.
\end{proof}

By analyzing the proof, we obtain the rest of the statement of Theorem \ref{thm:welldone}.

\begin{theorem}\label{thm:order}
The entropy spectrum does not contain a nonincreasing sequence corresponding to distinct subsurfaces. In particular, it is reverse well-ordered, and multiplicities are bounded. 
\end{theorem}
\begin{proof}
This follows from the proof of the previous result. By contradiction, let $\{Y_i\}$ be a sequence of distinct subsurfaces such that 
\[
h_{Y_{i+1}}\geq h_{Y_i}
\]
for all $i$. As it is monotone increasing, it converges to its supremum $h \in ]0,1]$. (Note that we could exclude $1$ by the gap result, but this plays no role here.) By the previous proof, it contains a subsequence with a maximal $Y$ subsurface in each $Y_i$ with $h_{Y_i}$ which converges to $h_Y=h$. But, since $h_{Y_i} > h_Y=h$ (see Lemma \ref{lem:strict} from the preliminaries), for each $i$, $h$ is not a supremum.
\end{proof}

In particular, the {\it proper} entropy spectrum, that is the set of entropies of proper subsurfaces, has a largest element. As such, for a closed surface $X$, we set 
\[
h_*(X):=\max\bigl(\mathcal E(X)\setminus\{1\}\bigr)
\]
and call it the largest proper entropy of $X$. It can be interpreted geometrically. 

\begin{corollary}\label{cor:proper}
The largest proper entropy of $X$ satisfies
\[
h_*(X)=
\max_{\substack{\alpha\in\mathcal G^0(X)\\
\alpha\ \mathrm{non\mbox{-}separating}}}
h_{X\setminus\alpha}.
\]
Conversely, if $Y$ is a subsurface of $X$ satisfying $h_Y = h_*(X)$, then $Y = X\setminus \alpha$ for some non-separating $\alpha$. 
\end{corollary}

\begin{proof}
By the previous theorems, $h_*(X)$ is well-defined and realized as the entropy of a connected subsurface $Y\subsetneq X$, that is $h_Y=h_*(X)$. By Lemma~\ref{lem:strict}, $Y$ cannot be properly contained in another proper connected subsurface. We now show that 
\[
Y=X\setminus\alpha
\]
for some non-separating simple closed geodesic $\alpha$. To see this, consider a curve $\alpha$ which is component of $\partial Y$. If $\alpha$ is non-separating, the connected surface $X\setminus\alpha$ contains $Y$, and maximality gives
$Y=X\setminus\alpha$. If $\alpha$ is separating, then $Y$ is one of the two components of $X\setminus \alpha$. Both components have positive genus, and hence there exists a non-separating curve $\beta$ on $X\setminus Y$. But now $X\setminus \beta$ strictly contains $Y$, and by maximality we have reached a contraction. 
\end{proof}

The equivalent of the largest proper entropy for the length spectrum would be the systole. Note that, unlike for the systole, there is only one topological type of a subsurface realizing the largest proper entropy. 

We end the paper with a few remarks.

\begin{remark}[The pants spectrum]
Our results identify, exactly, accumulation points in $\E(X)$, showing namely in which way it fails to be a discrete set. This means that certain natural subsets of $\E(X)$ are discrete subsets. 

For instance, consider the set $\E^P(X)$ of entropies of pairs of pants which lie in $X$. As there are only finitely many up to given boundary length, any infinite sequence of distinct pants must have systole going to infinity, and so the set of entropies must converge to $0$. Hence, as multisubset of $]0,1]$, $\E^P(X)$ is discrete. 

\end{remark}

\begin{remark}[Genus $2$]
If $X$ is of genus $2$, then the limited number of topological types of subsurfaces make the analysis even easier. All proper connected surfaces are either pairs of pants, one-holed tori, or the complement of a single non-separating curve. Thus, by Theorem \ref{thm:accumulation}, accumulation points correspond exactly to entropies of its one-holed tori subsurfaces. And any value which is not an accumulation point corresponds either to the entropy of a pair of pants or the complement of a single non-separating curve. As the former form a discrete subset, the non-discreteness only comes from the latter.\end{remark}

\begin{remark}[Non-filling geodesics and the entropy gap]\label{rem:LS}
There is another interpretation of $h_*(X)$, which follows from work of Lenzhen and Souto \cite{LenzhenSouto}. 

Let $T^1X$ denote the unit tangent bundle of $X$, and let $\NN(X)\subset T^1X$ be the set of unit tangent vectors whose corresponding complete geodesics are non-filling. Thus a vector belongs to $\NN(X)$ if the complete geodesic it determines does not intersect some essential simple closed geodesic transversely.

For a simple closed geodesic $\alpha$, let $\widehat\alpha\subset T^1X$ denote the set of vectors whose corresponding complete geodesics do not intersect $\alpha$ transversely. Lenzhen--Souto
\cite[Lemma~3.3]{LenzhenSouto} show that
\[
\dim_H(\widehat\alpha)=1+2h_{X\setminus\alpha}.
\]
Their statement is given in the projectivized tangent bundle but passing to the unit tangent bundle does not change Hausdorff dimension.

Now since
\[
\NN(X)=\bigcup_{\alpha\in\mathcal G^0(X)}\widehat\alpha
\]
is a countable union, it follows that
\[
\dim_H\NN(X)
=
1+2\sup_{\alpha\in\mathcal G^0(X)}h_{X\setminus\alpha}.
\]
By our description of the entropy spectrum, this supremum is its largest proper value. Using the notation $h_*(X)$ this becomes
\[
\dim_H\NN(X)=1+2h_*(X),
\]
or equivalently,
\[
3-\dim_H\NN(X)=2(1-h_*(X)).
\]
Thus the first gap in the entropy spectrum is exactly half the Hausdorff codimension of the set of complete non-filling geodesics.
\end{remark}

{\it Addresses:}\\
The Graduate Center and Hunter College, CUNY, NY, New York, USA.\\
Department of Mathematics, University of Fribourg, Switzerland\\

{\it Emails:}\\
abasmajian@gc.cuny.edu\\
hugo.parlier@unifr.ch\\

\end{document}